\newif\ifuseamsart
\useamsarttrue

\newif\ifdraft
\drafttrue

\ifuseamsart
	\documentclass{amsart}
	\usepackage{subcaption}
\else
	\documentclass[11pt]{article}
	\usepackage{amsmath,amsxtra,amssymb,latexsym, amscd,amsthm}
	\def\keywords{\vspace{-.5em}
		{\textit{Keywords}:\,\relax%
	}}
	
	\def\subjclass{\vspace{.5em}
		{\textit{2020 Mathematics Subject Classification.}:\,\relax%
	}}

	\newcommand\blfootnote[1]{%
		\begingroup
		\renewcommand\thefootnote{}\footnote{#1}%
		\addtocounter{footnote}{-1}%
		\endgroup
	}
\fi

\usepackage{floatpag}
\floatpagestyle{plain}

\usepackage{longtable}
\usepackage{stix}
\usepackage{xspace}
\usepackage{subcaption}
\usepackage{enumerate}
\usepackage{tikz}
\usetikzlibrary{patterns}
\usepackage{verbatim}
\usetikzlibrary {positioning}
  \usetikzlibrary{arrows,topaths}
   \usetikzlibrary{calc,intersections}
\usepackage{graphicx}
\usetikzlibrary{positioning, fadings, backgrounds}
\usetikzlibrary{shapes.geometric}

\newcommand{\red}{\xspace\mbox{\textcolor{red}{$\pmb\rightarrow$}}\xspace}
\newcommand{\blue}{\xspace\mbox{\textcolor{cyan}{$\pmb\rightarrow$}}\xspace}
\newcommand{\black}{\xspace\mbox{\textcolor{black}{$\pmb\rightarrow$}}\xspace}

\makeatletter
\newcommand\notsotiny{\@setfontsize\notsotiny{8}{8}}
\makeatother

\newcommand\Tstrut{\rule{0pt}{2.6ex}}         
\newcommand\Bstrut{\rule[-0.9ex]{0pt}{0pt}}   

 \newcommand{\LocationAndLength}{\mbox{$M_P$}\xspace}
 \newcommand{\FaceGraph}{key graph\xspace}

 \newcommand{\Cover}{\mbox{\textcolor{cyan}{$\pmb\rightarrow$}}}
 \newcommand{\Sharp}{\mbox{\textcolor{red}{$\pmb\rightarrow$}}}

\usetikzlibrary{decorations.markings}

\usepackage{float}
\usepackage{url}

\newtheorem{theorem}{Theorem}[section]
\newtheorem{lemma}[theorem]{Lemma}
\newtheorem{proposition}[theorem]{Proposition}
\newtheorem{corollary}[theorem]{Corollary}

\usepackage{float} 
\usepackage{stix} 
\usepackage{tikz}
\usepackage{caption}
\usepackage{enumitem}
\usepackage{xhfill}

\tikzstyle{startstop} = [rectangle, rounded corners, minimum width=3cm, minimum height=1cm,text centered, draw=black, fill=red!30]
\tikzstyle{process} = [rectangle, minimum width=3cm, minimum height=1cm, text centered, draw=black, fill=orange!30]
\tikzstyle{decision} = [diamond, minimum width=3cm, minimum height=1cm, text centered, draw=black, fill=yellow!30]
\tikzstyle{arrow} = [thick,->,>=stealth]
\title{The Schrijver system of the length polyhedron of an interval order}

\ifuseamsart

	\author{Andr\'{e} E. K\'{e}zdy}
	\address{Department of Mathematics, University of Louisville, Louisville, Kentucky, U.S.A.}
	\email{kezdy@louisville.edu}

	\author{Jen\H{o} Lehel}
	\address{Alfr\'ed R\'enyi Institute of Mathematics, Budapest, Hungary} 
	\email{lehelj@renyi.hu}
\else
	\author{Andr\'{e} E. K\'{e}zdy \\
		\small Department of Mathematics\\[-0.8ex]
		\small University of Louisville\\[-0.8ex]
		\small Louisville, Kentucky, U.S.A.\\[-0.8ex]
		\small\tt kezdy@louisville.edu\\
		\\
		Jen\H{o} Lehel \\
		\small  Alfr\'ed R\'enyi Institute of Mathematics\\[-0.8ex] 
		\small  Budapest, Hungary\\[-0.8ex]
		\small\tt lehelj@renyi.hu}  
\fi
	
\begin{document}

 \tikzstyle{A} = [circle,fill=black!,minimum width=.1em, inner sep=.1em]
 \tikzstyle{B} = [circle,fill=cyan,minimum width=.15em, inner sep=.2em]
 \tikzstyle{W} = [circle,fill=red!,minimum width=.15em, inner sep=.2em]
 \tikzstyle{Yr} = [rectangle,fill=red!,minimum width=.4em, inner sep=.4em] 
 \tikzstyle{Yb} = [rectangle,fill=cyan!,minimum width=.4em, inner sep=.4em] 
 \tikzstyle{X} = [circle,draw=black!,minimum width=2pt, inner sep=2pt]
 
\ifuseamsart
\else
	\maketitle
\fi

\begin{abstract}
	The length polyhedron of an interval order $P$
	is the convex hull of integer vectors representing the interval lengths in possible interval representations of $P$ in which all intervals have integer endpoints.
	This polyhedron is an integral translation of a polyhedral cone, with its apex corresponding to the canonical interval representation of $P$
	(also known as the minimal endpoint representation).
	
	In earlier work, we introduced an arc-weighted directed graph model, termed the \FaceGraph, inspired by this canonical representation. We showed that cycles in the \FaceGraph correspond, via Fourier-Motzkin elimination, to inequalities that describe supporting hyperplanes of the length polyhedron. 
	These cycle inequalities derived from the \FaceGraph form a complete system of linear inequalities defining the length polyhedron.
	By applying a theorem due to Cook, we establish here that this system of inequalities is totally dual integral (TDI).
	
	Leveraging circulations, total dual integrality, and the special structure of the \FaceGraph, our main theorem demonstrates that a cycle inequality is a positive linear combination of other cycle inequalities if and only if it is a positive integral linear combination of smaller cycle inequalities (where `smaller' here refers a natural weak ordering among these cycle inequalities).
	This yields an efficient method to remove redundant cycle inequalities and ultimately construct the unique minimal TDI-system, also known as the Schrijver system, for the length polyhedron. 
	Notably, if the \FaceGraph contains a polynomial number of cycles, this gives a polynomial-time algorithm to compute the Schrijver system for the length polyhedron.
	
	Lastly, we provide examples of interval orders where the Schrijver system has an exponential size.
	\ifuseamsart
	\fi

\end{abstract}
\ifuseamsart
	\keywords{interval order, canonical representation, key graph,  Fourier-Motzkin elimination, cycle inequalities, length polyhedron, convex cone, totally unimodular  matrix, total dual integrality, graph circulation, Schrijver system}
	\subjclass[2020]{06A06, 90C27, 05C62, 52B05, 05C20}
	\maketitle
\else
	
    \blfootnote{
   
    \keywords{interval order, canonical representation, key graph,  Fourier-Motzkin elimination, cycle inequalities, length polyhedron, convex cone, 
    	totally unimodular  matrix, total dual integrality, graph circulation, Schrijver system}
    
    \subjclass{06A06, 90C27, 05C62, 52B05, 05C20}}
\fi

\section{Introduction} 
In this paper, we continue our investigation of the length polyhedron of an interval order, building on earlier work \cite{LengthPolyhedronPaper}. A partially ordered set $P=([n],\prec)$ is called an interval order if there exists an assignment of compact intervals from $\mathbb{R}$ to elements of $[n]$, where each $x \in [n]$ is assigned an interval $I_x = [\ell_x, r_x] \subset \mathbb{R}$ such that $x \prec y$ if and only if $r_x < \ell_y$. 
For technical convenience and without loss of generality, we assume that $r_x + 1 \leq \ell_y$, for all comparable pairs $x \prec y$.
This choice reflects our focus on finite interval orders and their interval representations having integer interval endpoints.

The length polyhedron has been the subject of much research.  For example, it appears 
in Fishburn's book \cite{FishburnBook} where he
utilizes linear algebra and polyhedral techniques to explore properties of interval orders.  His focus
is primarily on the {\em interval count problem}.  This problem seeks to determine
the smallest number $k$ of distinct interval lengths required to represent a given interval order. 
Interval orders that can be represented with a single interval length ($k=1$)
are precisely the semiorders, which are well-studied and have well-known characterizations (see \cite{ScottSuppes,TrotterBook}). 

For $k\geq 2$, however, the interval count problem remains largely unresolved, with no general results available, even in the simplest case $k=2$. 
Isaak \cite{Isaak1993,Isaak2009} proposes a directed graph model to address this, though it is different from our model.
Progress also has been made by Doignon and Pauwels \cite{DoignonPauwels}, who additionally 
contribute applications to knowledge space theory \cite{DoFa} 
and mathematical psychology \cite{DoDuFa}.

Non-negative interval representations of an interval order $P$ with $n$ elements are characterized by the following system of linear inequalities.
 For every $x,y\in [n]$,
 \begin{eqnarray}\label{LOC}
\left\{
\begin{array}{rcll}
-\ell_x&\leq& 0,\\
r_x-\ell_y&\leq& -1\  & \text{\ if $x\prec y$},\ \\
\ell_x-r_y\  &\leq& 0& \text{\ if $x\not\prec y$}.
\end{array}
\right.
\end{eqnarray}
Each non-negative interval representation of $P$ corresponds to a vector in $\mathbb{R}^{2n}$ encoding a feasible solution of this linear system.  These feasible vectors form a convex polyhedron, which we denote $D_P$. Doignon and Pauwels \cite{DoignonPauwels} studied this polyhedron, determining its facet-defining inequalities and proving that it is
a pointed affine cone.

Interval count problems motivate our
study of a related polyhedron, {\em the length polyhedron} of an interval order $P$, denoted $Q_P$.  This polyhedron is
the convex hull of integer vectors representing the interval lengths in possible 
interval representations of $P$ in which all intervals have integer endpoints.
The components of the integer vectors of $Q_P$ correspond to interval lengths; we refer to these 
as $\rho_x = r_x - \ell_x$, for each $x \in [n]$.  In \cite{LengthPolyhedronPaper} we demonstrate
that $Q_P$ is a pointed affine cone and we compute its Hilbert basis.
In that investigation we introduce a novel directed graph, the \FaceGraph of $P$, 
which efficiently models the results of the Fourier-Motzkin process used
to derive $Q_P$ from (\ref{LOC}).  The length polyhedron is defined by the linear system of cycle inequalities associated with the directed cycles of the \FaceGraph (Theorem \ref{QPdef}). 
We prove here
(Proposition \ref{CYCTDI}), applying a theorem due to Cook (Theorem \ref{CookTDI}), that these cycle inequalities yield a totally dual integral (TDI) system of inequalities.  

Section \ref{PrelimSection} introduces these two central structures of an interval order that 
form the foundation of our analysis: the length polyhedron and the \FaceGraph. 
The length polyhedron, $Q_P$, provides a geometric framework to study the lengths of intervals in representations of an interval order $P$.  
In parallel, the \FaceGraph, $G_P$, serves as a combinatorial model for the inequalities defining the length polyhedron. 
This arc-colored, arc-weighted directed graph establishes a conceptual bridge between the length polyhedron and the interval order, offering deeper algorithmic insights.  We rely heavily on the structure of \FaceGraph and its circulations in this work.

Our objective here is to determine the minimal TDI linear system for the length polyhedron of an interval order. 
This linear system, known as the Schrijver system, is unique for the length polyhedron, 
a consequence of a general theorem due Schrijver (see Theorem \ref{SchrijverUniqueTDI}). Schrijver systems are 
closely related to min-max theorems in combinatorial optimization. 
We define Schrijver systems in Section \ref{PrelimSection}. 
Our presentation there reflects the integer programming origins (see the classic text by Schrijver \cite{Schrijver})
that inspired our work.
In the literature, Schrijver 
systems have been characterized for various polyhedra, 
including $b$-matchings \cite{CookandPulleyblank}, odd-join polyhedra \cite{Sebo}, 
$T$-cut polyhedra \cite{Rizzi}, and the flow cone of series-parallel graphs \cite{BarbatoEtAl}.

This paper advances the understanding of the length polyhedron, focusing on the computation of its Schrijver system. 
Because the cycle inequalities derived from cycles of the \FaceGraph yield a TDI system,
the unique Schrijver system may be obtained by discarding from this system
cycle inequalities that can be expressed as a non-negative integral linear combination of others (Theorem \ref{UniqueMinimalTDI}).
We show in Theorem \ref{integerNonNegativeLinearCombinationForInequalities} that all
redundant cycle inequalities are necessarily
integral linear combinations of smaller cycle inequalities with respect to the weak ordering of the cycle inequalities defined in Section \ref{weak}.  

Algorithmically, 
the ordered list of all cycle inequalities, referred as to the {\it Weak-list}, is first constructed from the \FaceGraph. 
The {\it Schrijver list}, which contains the minimal TDI system, is then populated in stages 
testing inequalities following the weak order.  Irredundant inequalities are added as they are detected, applying Theorem \ref{WeakOrderSignificance} which
guarantees that a redundant cycle inequality may be recognized by testing (via a single linear program) whether it 
is a positive linear combination of smaller inequalities.
Graph circulations are introduced in Section  \ref{CyclesSection} as the basic graph theory tool 
to detect and manipulate dependencies among cycle inequalities.
Notably, if the \FaceGraph contains a polynomial number of cycles, 
this gives a polynomial-time algorithm to compute the Schrijver system.

It is also worth noting that neither total unimodularity (Theorem \ref{LengthPlusLocationIsTUM})
nor TDI (Proposition \ref{CYCTDI}) are sufficient to prove our
main theorem (Theorem \ref{integerNonNegativeLinearCombinationForInequalities})
which states that a redundant cycle inequality is a non-negative {\em integral} 
linear combination of smaller cycle inequalities.  
Section \ref{CyclesSection} introduces graph theory tools essential for establishing structural properties used later.

Section \ref{mainproof} includes a constructive proof of the main Theorem \ref{integerNonNegativeLinearCombinationForInequalities}. 
Finally, we mention implementation considerations and provide an example
to illustrate that the Schrijver system of an interval order may have exponential size.

For the reader's benefit, we include an example of an interval order 
to illustrate a canonical representation, 
\FaceGraph, cycle inequalities, and a Shrijver system. 
We chose, for this purpose, to use the interval order in Example 1 of Doignon and Pauwels \cite{DoignonPauwels}.
For expedience, we identify interval orders with their ascent sequences (see Bousquet-Mélou et al. \cite{ascent2010}).  So, to introduce this example's interval order $P$ corresponding to the ascent sequence $(0, 1, 2, 2, 0, 2, 2, 3)$ we write
$P=(0, 1, 2, 2, 0, 2, 2, 3)$.
The canonical representation and \FaceGraph for  $P$ are displayed in Figure \ref{keygraphexample}.
Table \ref{ExampleTDI} presents the ten inequalities of the Schrijver system together with cycles of the \FaceGraph $G_P$ generating these inequalities.

\section{Preliminaries} 
\label{PrelimSection}
In this section, we introduce the two central structures associated with an interval order $P$ that form
the foundation of our analysis: the length polyhedron, $Q_P$, and the \FaceGraph, $G_P$. 
We begin with definitions and theorems that underpin the application of integer programming principles to $Q_P$. Further elaboration appears in Section \ref{SchrijverSection}. After summarizing these foundational tools, we present the canonical representation of an interval order and describe the \FaceGraph model, drawing on relevant prior results (see \cite{LengthPolyhedronPaper}) that frame the core object of this study: the length polyhedron.

 \subsection{Integer programming tools}
 \label{IntegerProgramming}
A matrix is {\em totally unimodular} if all its square submatrices have a determinant equal to $0$, $1$, or $-1$.
The following reformulation of a theorem due to Hoffman and Kruskal \cite{HoffmanKruskal} explains the significance of totally unimodular matrices.
\begin{theorem} [Hoffman and Kruskal \cite{HoffmanKruskal}, see also \cite{Schrijver}, Corollary 19.2b] \label{TUM} An integral matrix $A$ is totally unimodular
	if and only if for all integral vectors $\boldsymbol{b}$ and $\boldsymbol{c}$ both sides of the linear programming duality
	equation
	$$\max\{\boldsymbol{c}\boldsymbol{x} : \boldsymbol{x} \geq \boldsymbol{0}, A\boldsymbol{x} \leq \boldsymbol{b}\}
	= \min\{\boldsymbol{y}\boldsymbol{b} : \boldsymbol{y} \geq \boldsymbol{0}, \boldsymbol{y}A = \boldsymbol{c} \}
	$$
	are achieved by integral vectors $\boldsymbol{x}$ and $\boldsymbol{y}$, if they are finite.
\end{theorem}

Ghouila-Houri \cite{G-H} proved a useful combinatorial characterization of total unimodularity.
A $\{0,\pm 1\}$-matrix matrix $A$ has an {\em equitable bicoloring} if its columns can be partitioned into blue columns and red columns
so that, for every row of $A$, the sum of the entries in the blue columns differs from the sum of the entries in the red columns by at most one.  

\begin{theorem} [Ghouila-Houri \cite{G-H}, see also \cite{Schrijver}, Theorem 19.3] \label{Ghouila-Houri} A $\{0,\pm 1\}$-matrix $A$ is totally unimodular if and only if
	every submatrix of $A$ has an equitable bicoloring.
\end{theorem}

A linear system $A\boldsymbol{x} \leq \boldsymbol{b}$ is {\em totally dual integral} (TDI) if the minimum in the LP-duality equation
$\max\{\boldsymbol{w}^T\boldsymbol{x} : A\boldsymbol{x} \leq \boldsymbol{b} \} = \min\{\boldsymbol{b}^T \boldsymbol{y} : A^T \boldsymbol{y} = \boldsymbol{c}, \boldsymbol{y}\geq 0\}$ has an integer optimal solution $\boldsymbol{y}$ for all
integer vectors $\boldsymbol{c}$ for which the optimum is finite.

The question whether total dual integrality of a linear system remains totally dual after eliminating a 
variable by Fourier-Motzkin Elimination is the subject of 
 the following technical tresult due to Cook \cite{Cook}.
\begin{theorem}[Cook \cite{Cook}] \label{CookTDI} Let $A \boldsymbol{x} \leq \boldsymbol{b}$ be a TDI-system. 
	If each coefficient of the variable $x_1$ is either $0$, $1$, or $-1$, then the system obtained by eliminating $x_1$
	by Fourier-Motzkin elimination is also TDI.
\end{theorem}


A totally dual integral system is a {\em minimal TDI-system} if any proper
subsystem which defines the same polyhedron is not TDI.
Equivalently, a totally dual integral system is a minimal TDI-system if and only if
every constraint in the system determines a supporting hyperplane and is
not the non-negative integral combination of other constraints in the system.
Schrijver proved the following theorem.
\begin{theorem}[Schrijver \cite{SchrijverTDI}] \label{SchrijverUniqueTDI}
A full-dimensional rational polyhedron is determined
by a unique minimal TDI-system $A \boldsymbol{x} \leq \boldsymbol{b}$ of linear
inequalities with $A$ integral.  
\end{theorem}
The unique minimal TDI-system is now referred to as a {\em Schrijver system},
following the nomenclature introduced by Cook and Pulleyblank \cite{CookandPulleyblank}.
\subsection{The \FaceGraph and the length polyhedron} 
\label{KeygraphSection}
 Let $P=([n],\prec)$ be an interval order. The \FaceGraph $G_P$  is a colored arc-weighted directed graph whose arcs are defined by certain pairs of the canonical representation of $P$. 

Greenough proved  \cite[Theorem 2.4]{Greenough}  that the {\it magnitude representation} of  
$P$, the interval representation of $P$ minimizing the number of distinct endpoints, is essentially unique in the sense that, disregarding the `physical' locations, the endpoint incidences agree in every minimal endpoint representation. We define the {\it canonical representation} of $P$ as a magnitude representation on integral interval endpoints $0,1,\ldots. m-1$, where $m$ is the {\em magnitude} of $P$. 

In relation to the fundamental inequalities in (\ref{LOC}) we define the concept of `slack' as follows.
Let $\mathcal{R}=[\ell_x,r_x]_{x\in[n]}$  be an integral representation of $P$.
For  $x,y\in [n]$ the \emph{slack in $\mathcal{R}$}  of the pair $(x,y)$  is defined to be
\[s_{\mathcal{R}}(x,y)= 
\left\{
\begin{array}{lcl}
	\ell_y-r_x-1 &\text{ if}  & x\prec y\\
	r_y-\ell_x & \text{ if}  &    x\|y\\
	\rho_x= r_x-\ell_x& \text{ if}  &    x=y
\end{array} \quad .
\right.
\]

Note that $s_{\mathcal{R}}(x,y)\geq 0$, and it is defined for every pair $(x,y)$, unless $y\prec x$.

A pair $(x,y)$ is a {\it slack zero pair} provided that $s_{\mathcal{C}}(x,y)=0$,
where $C$ is the canonical representation of $P$.
A slack zero pair $(x,y)$ is a {\em slack zero sharp pair} if $x \| y$; it is a {\em slack zero cover pair} if $x\prec y$.
Additionally,
by definition, $(x,x)$ is a slack zero pair if and only if the interval assigned to $x$ in $C$ has length zero.
Figure \ref{slackzeropairs} displays slack zero pairs in the canonical representation. 

 According to the Slack theorem   \cite[Theorem 3.6]{LengthPolyhedronPaper}, if $C$ is the canonical representation of $P$, then $s_{\mathcal{C}}(x,y)\leq s_{\mathcal{R}}(x,y)$, for every $x,y$, $y\not\prec x$.  As a corollary, the property of being a slack zero pair in $P$ is independent of the particular representation of $P$.

\begin{figure}[htp]
\begin{center}
\begin{tikzpicture}[scale=.8]
  \foreach \i in {-1,...,1}{
   \draw[line width=.05em](\i,1)--(\i,2);}  
     \node[label=below:{$(x,y)$ slack zero sharp pair}]()at(0,0.25){};

\node[A,label=left:x](L)at(0,1.6){}; \node(R)at(1.6,1.6){};\draw[line width=.05em](L)--(R);
\node[A,label=right:y](R)at(0,1.25){}; \node[](L)at(-1.5,1.25){};\draw[line width=.05em](L)--(R);

\end{tikzpicture}\hskip2cm
\begin{tikzpicture}[scale=.8]
  \foreach \i in {-2,...,0}{
   \draw[line width=.05em](\i,1)--(\i,2);}  
     \node[label=below:{$(x,y)$ slack zero cover pair}]()at(0,0.25){};
 \node[label=above:{$(w,w)$ slack zero pair}]()at(-0.42,-0.25){};
\node[A](L)at(-2,1.6){}; 
\node[label=left:w](R)at(-2,1.6){};
\node[A,label=left:y](L)at(0,1.6){}; 
\node[](R)at(1.1,1.6){};\draw[line width=.05em](L)--(R);
\node[A,label=right:x](L)at(-1,1.25){}; \node[](R)at(-2.5,1.25){};\draw[line width=.05em](L)--(R);
\end{tikzpicture}
\end{center}
\caption{Slack zero pairs shown in part of a canonical representation}
\label{slackzeropairs}
\end{figure}
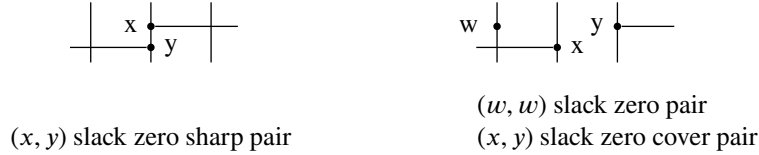
 
Let  $\mathcal{R}=\{\ell_i,r_i\}_{i\in[n]}$ be an interval representation of $P$.
The $2n$ endpoint variables define
the $\rho$-vector $(\rho_1,\ldots,\rho_n)\in\mathbb{R}^n$, where $\rho_i=r_i-\ell_i$, $1\leq i\leq n$. Doignon and Pauwels's result  \cite[Theorem 4]{DoignonPauwels} (see also  \cite[Theorem 4.2]{LengthPolyhedronPaper}) implies that the existence of an interval representation of $P$ with a given $\rho$-vector can be stated as the feasibility of 
the system
\begin{eqnarray}\label{LRO}
\left\{
\begin{array}{rcll}
 \ell_x+\rho_x+1&\leq& \ell_y\  & \text{\ if $(x,y)$ is a slack zero cover pair},\ \\
\ell_x&\leq& \ell_y+\rho_y\  & \text{\ if $(x,y)$ is a slack zero sharp pair},\\
-\rho_x&\leq& 0 &  \text{\ if $(x,x)$ is a slack zero  pair}.
\end{array}
\right.
\end{eqnarray}

The arcs of the \FaceGraph of $P$ are defined by  slack zero cover pairs and slack zero sharp pairs.
The {\it \FaceGraph} $G_P$ is a colored arc-weighted directed graph that has vertex set $V=\{\rho_1,\rho_2,\ldots,\rho_n\}$, and 
$\rho_x\black\rho_y$ is an arc (or a loop) in $G_P$ if and only if $(x,y)$ is a slack zero pair of $P$.
The color and the weight of $\rho_x\black\rho_y$ indicate the relation of the variables $\rho_x,\rho_y$ in the corresponding inequalities of the linear system (\ref{LRO}) defining the representation polyhedra. There are 
red loops, and blue or red arcs between $\rho_x,\rho_y\in V$  with weights as follows:
\begin{eqnarray}\label{FaceGraphWeights}
\left\{	
\begin{array}{crclll}
	\text{arc} & \text{weight}\\  
 	\rho_x\blue \rho_y \ &\ell_x + \rho_x + 1 &\leq& \ell_y  & \text{if }   (x,y) \text{ is a slack zero cover pair}, \\ 
     \rho_x\red \rho_y\  &\ell_x &\leq& \ell_y + \rho_y & \text{if }  (x,y) \text{ is a slack zero sharp pair}, \\ 
      \rho_x\red \rho_x \ &- \rho_x &\leq& 0 & \text{if }  \text{$(x,x)$ is a slack zero  pair}.
\end{array}
\right.
\end{eqnarray}  

The sum of the arc weights along a directed cycle $C\subset G_P$ is an inequality  
containing only $\rho$-variables. Every such {\it cycle inequality}  has the form 
\begin{eqnarray}\label{CYC}
 \gamma \ +\sum_{i\in A}\rho_i \ \leq \ \sum_{j\in B}\rho_j\  ,
\end{eqnarray}
where $\gamma$ and $A,B\subset X$ are disjoint sets defined as follows.

  For every   directed $3$-path $T=(\rho_x\rightarrow \rho_y\rightarrow \rho_z)$  along a cycle $C$ of the \FaceGraph,  $y\in A$ if both arcs of $T$ are blue, $y\in B$ if both arcs of $T$ are red, and $y\notin A\cup B$, otherwise.
Furthermore, $\gamma$ in (\ref{CYC}) counts the number of  blue arcs along $C$. 
A  loop at $\rho_y$ is considered as a singleton red cycle $C=(\rho_y)$,  and it has weight $\langle 0\leq \rho_y\rangle$.

 Figure \ref{keygraphexample} presents the interval order $P=(0, 1, 2, 2, 0, 2, 2, 3)$ in Example 1 from Doignon and Pauwels \cite{DoignonPauwels}. 
The canonical representation and \FaceGraph for  $P$ are displayed in Figure \ref{keygraphexample}.
Table \ref{ExampleTDI} presents the ten inequalities of the Schrijver system together with cycles of the \FaceGraph $G_P$ generating these inequalities.
For example, by the cycle interpretation rules, the cycle inequality (4) defined by  the directed cycle $C=(\rho_1,\rho_2,\rho_7,\rho_4,\rho_8,\rho_6,\rho_5)$
is $\langle 3+\rho_ 2+\rho_7\leq \rho_5+\rho_6+\rho_8\rangle$.
The \FaceGraph $G_P$ has $25$ directed cycles (including four loops) corresponding to $17$ distinct inequalities.  Seven of these cycle inequalities are redundant.
For example, $\langle 0\le \rho_{3}+\rho_{4}\rangle$ and $\langle 3+\rho_{2}+\rho_{6}\le \rho_{3}+\rho_{5}+\rho_{7}+\rho_{8}\rangle$ are redundant.

\definecolor{edgeBlack}{RGB}{0,0,0}
\definecolor{edgeBlue}{RGB}{0,0,255}
\definecolor{edgeRed}{RGB}{255,0,0}

	\begin{figure}[H]
		\begin{center}
			\begin{tikzpicture}[scale=1.1]
				\node () at (7,3){};
				\foreach \i in {0,...,4}{
					\draw[line width=.02em](3+\i,4)--(3+\i,6);
					\node[]()at(3+\i,6.2){\tiny{\i}};   }
				
				\node[A,label=left:1](L) at (3,5) {};
				
				\node[A,label=left:6](L) at (5,5) {};
				\node[A](R) at (6,5) {};
				\draw[line width=.05em](L)--(R); 
				
				\node[A,label=left:7](L) at (5,4.5) {};
				\node[A](R) at (6,4.5) {};
				\draw[line width=.05em](L)--(R); 
				
				\node[A,label=left:5](L) at (3,5.5) {};
				\node[A](R) at (5,5.5) {};
				\draw[line width=.05em](L)--(R); 
				
				\node[A,label=left:8](L) at (6,5.5) {};
				\node[A](R) at (7,5.5) {};
				\draw[line width=.05em](L)--(R);
				
				\node[A,label=left:2]()at (4,4.5) {};
				\node[A,label=left:4]()at (7,5) {};
				\node[A,label=left:3]()at (7,4.5) {};
			\end{tikzpicture}
			\hskip2cm
			\begin{tikzpicture}[scale=.68]

				\node[A,label=above:{ $\rho_1$}](1) at (2.5,-1.2) {}; 
				\node[A,label=above left:{ $\rho_2$}](2) at (-0.5,0.0) {};
				\node[A,label=below left:{ $\rho_3$}](3) at (1,6) {};
				\node[A,label=below right:{ $\rho_4$}](4) at (4,6) {};
				\node[A,label=right:{ $\rho_5$}](5) at (5.5,0.0) {};
				\node[A,label=left:{ $\rho_6$}](6) at (0,3) {};
				\node[A,label=below:{ $\rho_7$}](7) at (2.5,1.2) {};
				\node[A,label=right:{ $\rho_8$}](8) at (5,3) {};

				\draw[line width=.1em,red,->](3)--(8);
				\draw[line width=.1em,red,->](4)--(8);
				\draw[line width=.1em,red,->](8)--(6);
				\draw[line width=.1em,red,->](5)--(1);
				\draw[line width=.1em,cyan,->](5)--(8);
				\draw[line width=.1em,red,->](6)--(5);
				\draw[line width=.1em,red,->](7)--(5);
				\draw[line width=.1em,cyan,->](2)--(6);
				\draw[line width=.1em,cyan,->](2)--(7);
				\draw[line width=.1em,cyan,->](1)--(2);
				\draw[line width=.1em,cyan,->](6)--(4);
				\draw[line width=.1em,cyan,->](6)--(3);
				\draw[line width=.1em,cyan,->](7)--(4);
				\draw[line width=.1em,cyan,->](7)--(3);
				\draw[red,line width=.1em,->] (1) edge[out=200, in=-65, looseness=16] (1);
				\draw[red,line width=.1em,->] (2) edge[out=155, in=-110, looseness=16] (2);
				\draw[red,line width=.1em,->] (3) edge[out=90, in=180, looseness=16] (3);
				\draw[red,line width=.1em,->] (4) edge[out=90, in=0, looseness=16] (4);
				\draw[red,line width=.1em,->] (3) edge[out=45, in=135, looseness=0.5]  [bend left] (4);
				\draw[red,line width=.1em,->] (4) edge[out=135, in=45, looseness=0.5] [bend left] (3);
				\draw[line width=.1em,red,->] (8)--(7);
			\end{tikzpicture}
		\end{center}
		\caption{The interval order $P=(0, 1, 2, 2, 0, 2, 2, 3)$.  Its canonical representation (left)   
			and its \FaceGraph $G_P$ (right).}
		\label{keygraphexample}
	\end{figure}
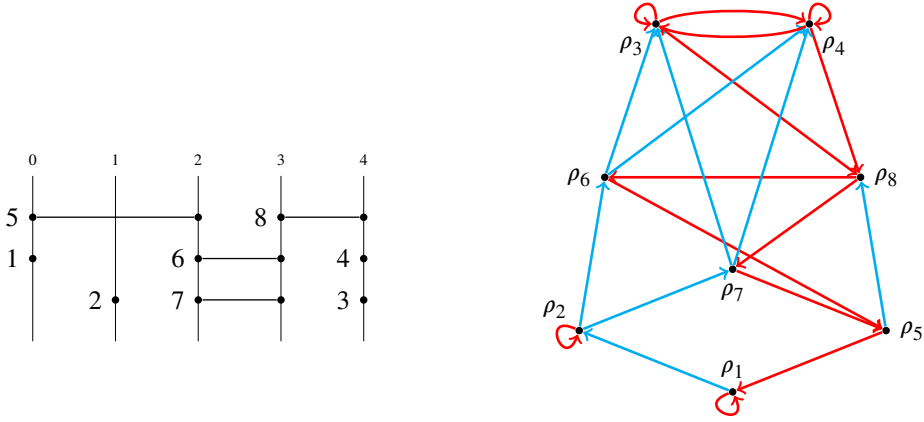
\begin{center}
	\begin{longtable}{r|c|l}
		\caption{All ten inequalities in the Schrijver system of the length polyhedron of the interval order $P=(0, 1, 2, 2, 0, 2, 2, 3)$. \label{ExampleTDI}}\\
		\hline \Tstrut
		\textbf{\#} & \textbf{facet-defining inequality} & \textbf{a corresponding \FaceGraph cycle} \\
		\hline
		
		\endfirsthead
		\multicolumn{3}{c}%
		{\tablename\ \thetable\ -- \textit{inequalities continued from previous page}} \\
		\hline \Tstrut
		\textbf{\#} & \textbf{facet-defining inequality} & \textbf{a corresponding \FaceGraph cycle} \\
		\hline
		\endhead
		\endfoot
		\hline
		\endlastfoot
		\Tstrut
		1 & $0 \leq \rho_{1}$ & $\Sharp  \rho_1 \Sharp $ \\ \Tstrut
		2 & $0 \leq \rho_{2}$ & $\Sharp  \rho_2 \Sharp $ \\ \Tstrut
		3 & $0 \leq \rho_{3}$ & $\Sharp  \rho_3 \Sharp $ \\ \Tstrut
		4 & $0 \leq \rho_{4}$ & $\Sharp  \rho_4 \Sharp $\\ \Tstrut
		5 & $1 \leq \rho_{8}$ & $ \Cover \rho_3 \Sharp  \rho_8 \Sharp  \rho_6 \Cover$ \\ \Tstrut
		6 & $1 \leq \rho_{6}$ & $ \Sharp  \rho_5 \Cover \rho_8 \Sharp  \rho_6 \Sharp$ \\ \Tstrut
		7 & $1 \leq \rho_{7}$ & $ \Sharp  \rho_5 \Cover \rho_8 \Sharp  \rho_7 \Sharp$ \\ \Tstrut
		8 & $2 + \rho_{2} \leq \rho_{5}$ & $ \Sharp  \rho_1 \Cover \rho_2 \Cover \rho_6 \Sharp  \rho_5 \Sharp$ \\ \Tstrut
		9 & $3 + \rho_{2}+\rho_{6} \leq \rho_{5}+\rho_{7}+\rho_{8}$ & $ \Sharp  \rho_1 \Cover \rho_2 \Cover \rho_6 \Cover \rho_3 \Sharp  \rho_8 \Sharp  \rho_7 \Sharp  \rho_5 \Sharp$ \\ \Tstrut
		10 & $3 + \rho_{2}+\rho_{7} \leq \rho_{5}+\rho_{6}+\rho_{8}$ & $ \Sharp  \rho_1 \Cover \rho_2 \Cover \rho_7 \Cover \rho_3 \Sharp  \rho_8 \Sharp  \rho_6 \Sharp  \rho_5 \Sharp$  \Bstrut
		
	\end{longtable}
\end{center}
 
We proved \cite{LengthPolyhedronPaper} that eliminating the left endpoint variables from the linear system (\ref{LRO}) by  Fourier-Motzkin elimination results in the equivalent linear system in variables $\rho_1,\ldots,\rho_n$ formed by the inequalities  (\ref{CYC}), one for each directed cycle $C$  of $G_P$.
  Our prior results about the length polyhedron of an interval order are summarized next (c.f. \cite[Theorem 5.1]{LengthPolyhedronPaper}).
 \begin{theorem}	
 \label{fromlength}{\em (Biro et al. \cite{LengthPolyhedronPaper})} \label{QPdef}  The {\em length polyhedron} $Q_P\subset \mathbb{R}^n$ is the set of all feasible $\rho$-vectors satisfying  the linear system of the inequalities (\ref{CYC}) for all directed cycles of the the \FaceGraph;  
 $Q_P$  is a full-dimensional affine cone with apex at the length vector of the canonical representation of $P$.
\end{theorem}

 \section{The Schrijver system}
 \label{SchrijverSection}

Let $P=([n], \prec)$ be an interval order and let $Q_P$ be its length polyhedron.
Recall that $Q_P$ is the projection of the conic polyhedron defined by the
linear system (\ref{LRO}), which may be presented this way:
\begin{eqnarray}\label{LRO2}
	\left\{
	\begin{array}{rcll}
		\ell_x+\rho_x-\ell_y&\leq& -1\  & \text{\ if $(x,y)$ is a slack zero cover pair}\ \\
		\ell_x - \ell_y -\rho_y&\leq& 0\  & \text{\ if $(x,y)$ is a sharp pair }\\
		-\rho_x&\leq& 0 & \text{\ if  $(x,x)$ is a slack zero pair}\\
		-\ell_x&\leq& 0 & \text{\ for every $x\in X.$}
	\end{array}
	\right. 
\end{eqnarray}
Write the linear system  (\ref{LRO2}) in matrix form as  $\LocationAndLength \boldsymbol{x} \leq {\boldsymbol{b}}_P$,
where \LocationAndLength is the $\{0,\pm 1\}$-matrix whose rows correspond the left-hand sides of the inequalities,
 its columns correspond to the $2n$ variables $\ell_1,\ldots,\ell_n,\rho_1,\ldots,\rho_n$ (in this order), and the entries are the coefficients of these variables.
The $0, -1$ entries of the vector ${\boldsymbol{b}}_P$ are determined by the right-hand sides of these inequalities.

\begin{theorem}  \label{LengthPlusLocationIsTUM} For any interval order $P=([n], \prec)$, the matrix \LocationAndLength is totally unimodular.
\end{theorem}
\begin{proof}
	We apply the characterization of totally unimodular matrices due to Ghouila-Houri \cite{G-H}. 
	Consider an arbitrary submatrix $A$ of \LocationAndLength
	formed by selecting certain rows and columns.  Color all columns of $A$ blue, except in the
	case in which columns corresponding to $\ell_x$ and $\rho_x$ are both selected to form $A$.
	In this exceptional case, color the column corresponding to $\ell_x$ blue and 
	the column corresponding to $\rho_x$ red.  We claim that this is an equitable bicoloring of the columns of $A$.
	To prove this, consider a row of $A$.  If this row corresponds to an inequality involving only one variable,
	then the blue sums and the red sums of the entries differ by at most one.  So, it suffices to consider
	expressions of type $\ell_x+\rho_x-\ell_y$ since the left-hand sides of the inequalities $	\ell_x+\rho_x-\ell_y\leq -1$ and $\ell_x - \ell_y -\rho_y\leq 0$ from (\ref{LRO2}) can be rewritten into this form.  Observe that if $\ell_x$ and $\rho_x$ are not
	both selected to form $A$, then this expression is equivalent to $\ell_x-\ell_y$ or $\rho_x-\ell_y$ and all
	of these terms are blue, hence their blue sum is in $\{0,\pm1\}$ while the red sum is zero.
	Similarly, if $\ell_x$ and $\rho_x$ are 
	both selected to form $A$, then $\ell_x+\rho_x-\ell_y$ has blue sum $0$ or $1$, depending upon whether
	the column corresponding to $\ell_y$ is chosen to form $A$, whereas its red sum is one.
	So $A$ has an equitable bicoloring.  Because $A$ is arbitrary, Theorem \ref{Ghouila-Houri} implies that 
	\LocationAndLength is totally unimodular.
\end{proof}

\begin{corollary} \label{TDI} For any interval order $P=([n], \prec)$ and any integral vector $\boldsymbol{b}$,
	the linear system $\LocationAndLength \boldsymbol{x} \leq \boldsymbol{b}$ is TDI.
\end{corollary}
\begin{proof}  We apply the Hoffman and Kruskal characterization \cite{HoffmanKruskal} of totally unimodular matrices stated here in Theorem \ref{TUM} combined with the  definition of total dual integrality.
Then Theorem \ref{LengthPlusLocationIsTUM} implies that the linear system $\LocationAndLength \boldsymbol{x} \leq \boldsymbol{b}$ is TDI.
\end{proof}

To obtain a linear system in terms of the $\rho$-variables equivalent to (\ref{LRO2})  we eliminate the variables 
 $\ell_x$, $x=1,\ldots,n$, from (\ref{LRO2})  by using Fourier-Motzkin elimination (FME). Let $N\boldsymbol{x} \leq \boldsymbol{b}$ be the  linear system containing the $n$ $\rho$-variables only. The feasible solutions of this system define the length polyhedron $Q_P$. 
 By Definition \ref{QPdef}, given for $Q_P$ in terms of the cycle inequalities, the linear system $N\boldsymbol{x} \leq \boldsymbol{b}$ is equivalent to the system (\ref{CYC}) for all directed cycles of the \FaceGraph $G_P$.
 Notice that the  inequalities of the form $-\ell_x\leq 0$  in (\ref{LRO2}) are meaningless for $Q_P$.

 \begin{proposition}
 \label{CYCTDI}
 The linear system $N\boldsymbol{x} \leq \boldsymbol{b}$ of the cycle inequalities in the \FaceGraph of an interval order is TDI. 
 \end{proposition}
 \begin{proof}
 Denote by $N_i\boldsymbol{x} \leq \boldsymbol{b_i}$ the linear system equivalent to (\ref{LRO2}) after eliminating $\ell_x$ for $x=1,2,...i$. Then  $N_n\boldsymbol{x} \leq \boldsymbol{b_n}$ is identical with 
 the system $N\boldsymbol{x} \leq \boldsymbol{b}$ of cycle inequalities. 
 
 We apply a result due to Cook \cite{Cook} that is stated here as Theorem \ref{CookTDI}. Because 
each coefficient of the variables in  a cycle inequality (\ref{CYC}) is either $0$, $1$, or $-1$,  the conditions required by Theorem \ref{CookTDI}
 are satisfied by $N_i\boldsymbol{x} \leq \boldsymbol{b_i}$, for every $i=1,\ldots, n$. The FME procedure  starts with a TDI-system (\ref{LRO2}), 
 by Corollary \ref{TDI}. Thus, repeated application of Theorem \ref{CookTDI}  implies  that the system $N\boldsymbol{x} \leq \boldsymbol{b}$ of cycle inequalities is TDI.
\end{proof}

\begin{theorem} \label{UniqueMinimalTDI} The length polyhedron  $Q_P$ of an interval order  $P$ 
has a unique Schrijver system obtained from the system of cycle inequalities in the \FaceGraph by discarding each
cycle inequality that can be expressed as a non-negative integral linear combination of other inequalities (not equivalent to the original inequality).
\end{theorem}
\begin{proof}
	By Proposition \ref{CYCTDI}, the system (\ref{CYC}) of cycle inequalities in the \FaceGraph is TDI.
	  As outlined in Theorem \ref{fromlength}, the length polyhedron $Q_P$ is defined by (\ref{CYC}) and it is a full-dimensional integral affine cone.  The claim now follows by Theorem \ref{SchrijverUniqueTDI}.
	\end{proof}
		Theorem \ref{UniqueMinimalTDI}  combined with our	main result Theorem \ref{WeakOrderSignificance} 
 answers the question how to compute the Schrijver system of the length polyhedron $Q_P$ of an interval order $P$. 
 	In fact, we prove in the second part of the paper  that redundant cycle inequalities are necessarily
	integral linear combinations of smaller cycle inequalities with respect to the weak ordering of the cycle inequalities defined in Section \ref{weak}.

As an example, we provide the Schrijver system for $P=(0, 1, 2, 2, 0, 2, 2, 3)$ (shown in Figure \ref{keygraphexample}).
Table \ref{ExampleTDI} contains the Schrijver system of the length polyhedron $Q_P$ together with the  cycles of the \FaceGraph $G_P$ generating the corresponding inequalities.

\section{Paths and cycles in the \FaceGraph}
\label{CyclesSection}

This section develops the structure of the \FaceGraph, focusing on the
properties needed in later sections: local structure (the diamond lemma), 
color-alternating directed paths, and circulations.

\subsection{The diamond lemma}
 Let $P=([n],\prec)$ be an interval order. 
 Recall that the {\it \FaceGraph}, $G_P$, is a colored, arc-weighted directed graph derived from $P$.
 Two arcs of the {\it \FaceGraph} with the same color and common tail (or with a common head) determine a gap in the canonical representation of $P$, an interval between consecutive interval endpoints.  
 The next lemma, dubbed the {\it diamond lemma}, details how such arcs (and their gaps) force other arcs in the {\it \FaceGraph}.
 This is a critical tool in detecting dependencies among cycle inequalities.

\begin{lemma} \label{diamond}
Let $\rho_w,\rho_x,\rho_y,\rho_z$ be  
distinct vertices of $G_P$.
\[
\begin{array}{lll}   
{\rm (I)}\quad  &\text{If } \rho_w\blue \rho_x, \rho_w\blue \rho_y, \text{ and }\rho_x\red \rho_z \text{ are arcs of }\ G_P,
	 \text{ then}\  \rho_y\red \rho_z \text{ is an arc of}\ G_P,\\
{\rm (II)}\quad  & \text{If } \rho_w\red \rho_x, \rho_w\red \rho_y,  \text{ and } \rho_x\blue \rho_z \text{ are arcs of }\ G_P,
	 \text{ then } \rho_y\blue \rho_z\text{ is an arc of } G_P,\\	
	{\rm (III)}\quad  &\text{If } \rho_w\red \rho_x, \rho_w\red \rho_y, \text{ and } \rho_z\red \rho_x \text{ are arcs of }\ G_P,
	\text{ then}\  \rho_z\red \rho_y\ \text{is an arc of}\ G_P,\\
{\rm (IV)}\quad  &\text{If}\  \rho_w\blue \rho_x, \rho_w\blue \rho_y, \text{ and }\rho_z\blue \rho_x \text{ are arcs of }\ G_P,
	\text{ then}\  \rho_z\blue \rho_y\ \text{is an arc of}\ G_P. \\
\end{array}
\]
\end{lemma}	


	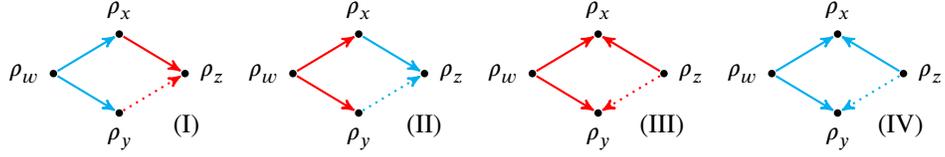
\begin{figure}[H]
	\begin{center}
		\begin{tikzpicture}[scale=.7, ->,>=stealth',auto,node distance=3cm,
				thick,main node/.style={circle,draw,font=\sffamily\small\bfseries}, node/.style={}]
				
					\node[A,label=left:$\rho_w$] at (0,.75) (1) {};
				\node[A,label=above:$\rho_x$] at (1.25,1.5) (2) {};
				\node[A,label=right:$\rho_z$] at (2.5,.75) (3) {};
				\node[A,label=below:$\rho_y$] at (1.25,0) (4) {};
				
				\path[draw, thick, every node/.style={font=\sffamily\small}]
				(1) edge [cyan] node [] {} (2)
				(1) edge [cyan] node [] {} (4)
				(2) edge  [red]  node [] {} (3)
				(4) edge [red,dotted] node [] {} (3);
				\node[label=left:(I)]()at (3.15,-0.25){};	
			\end{tikzpicture}
	\begin{tikzpicture}[scale=.7, ->,>=stealth',auto,node distance=3cm,
				thick,main node/.style={circle,draw,font=\sffamily\small\bfseries}, node/.style={}]
				
					\node[A,label=left:$\rho_w$] at (0,.75) (1) {};
				\node[A,label=above:$\rho_x$] at (1.25,1.5) (2) {};
				\node[A,label=right:$\rho_z$] at (2.5,.75) (3) {};
				\node[A,label=below:$\rho_y$] at (1.25,0) (4) {};
				
				\path[draw, thick, every node/.style={font=\sffamily\small}]
				(1) edge [red] node [] {} (2)
				(1) edge [red] node [] {} (4)
				(2) edge  [cyan]  node [] {} (3)
				(4) edge [cyan,dotted] node [] {} (3);
				\node[label=left:(II)]()at (3.2,-0.25){};	
			\end{tikzpicture}
					\begin{tikzpicture}[scale=.7, ->,>=stealth',auto,node distance=3cm,
				thick,main node/.style={circle,draw,font=\sffamily\small\bfseries}, node/.style={}]
				
					\node[A,label=left:$\rho_w$] at (0,.75) (1) {};
				\node[A,label=above:$\rho_x$] at (1.25,1.5) (2) {};
				\node[A,label=right:$\rho_z$] at (2.5,.75) (3) {};
				\node[A,label=below:$\rho_y$] at (1.25,0) (4) {};
				
				\path[draw, thick, every node/.style={font=\sffamily\small}]
				(1) edge [red] node [] {} (2)
				(1) edge [red] node [] {} (4)
				(3) edge  [red]  node [] {} (2)
				(3) edge [red,dotted] node [] {} (4);
				\node[label=left:(III)]()at (3.25,-0.25){};	
			\end{tikzpicture}
	\begin{tikzpicture}[scale=.7, ->,>=stealth',auto,node distance=3cm,
				thick,main node/.style={circle,draw,font=\sffamily\small\bfseries}, node/.style={}]
				
					\node[A,label=left:$\rho_w$] at (0,.75) (1) {};
				\node[A,label=above:$\rho_x$] at (1.25,1.5) (2) {};
				\node[A,label=right:$\rho_z$] at (2.5,.75) (3) {};
				\node[A,label=below:$\rho_y$] at (1.25,0) (4) {};
				
				\path[draw, thick, every node/.style={font=\sffamily\small}]
				(1) edge [cyan] node [] {} (2)
				(1) edge [cyan] node [] {} (4)
				(3) edge  [cyan]  node [] {} (2)
				(3) edge [cyan,dotted] node [] {} (4);	
				\node[label=left:(IV)]()at (3.25,-0.25){};
			\end{tikzpicture}
	\end{center}
	\caption{The Diamond Lemma}
	\end{figure}

\begin{proof} 	
	 For (I) and (IV),  $\rho_w\blue \rho_x$ and $\rho_w\blue \rho_y$ are arcs of $G_P$, hence
	$\ell_x = r_w + 1 = \ell_y$.  
	
	If   $\rho_x\red \rho_z$ is in $G_P$, then $r_z = \ell_x= \ell_y$ which implies $\rho_y\red \rho_z$ is  in $G_P$.
	
	If   $\rho_z\blue \rho_x$ is in $G_P$, then $r_z = \ell_x$ $r_z = \ell_y$ which implies $\rho_z\blue \rho_y$ is in $G_P$.
	
	\begin{figure}[htp]
\begin{center}
\begin{tikzpicture}[scale=.8]
  \foreach \i in {-1,...,0}{
   \draw[line width=.05em](\i,.75)--(\i,2);}  
     \node[label=above:{$\ell_x=\ell_y$}]()at(0,0){};

\node[A](L)at(-1,1.6){}; \node[label=left:w](R)at(-1.8,1.6){};\draw[line width=.05em](L)--(R);
\node[A](L)at(0,1.6){}; \node[label=right:x](R)at(1.1,1.6){};\draw[line width=.05em](L)--(R);
\node[A](L)at(0,1.25){}; \node[label=right:y](R)at(1.1,1.25){};\draw[line width=.05em](L)--(R);
\node[A](L)at(0,.9){}; \node[label=left:z](R)at(-1.5,.9){};\draw[line width=.05em](L)--(R);
\node[label=left:(I)]()at (2,.75){};	
\end{tikzpicture}\hskip2cm
\begin{tikzpicture}[scale=.8]
  \foreach \i in {-1,...,0}{
   \draw[line width=.05em](\i,.75)--(\i,2);}  
     \node[label=above:{$\ell_x=\ell_y$}]()at(0,0){};

\node[A](L)at(-1,1.6){}; \node[label=left:w](R)at(-1.8,1.6){};\draw[line width=.05em](L)--(R);
\node[A](L)at(0,1.6){}; \node[label=right:x](R)at(1.1,1.6){};\draw[line width=.05em](L)--(R);
\node[A](L)at(0,1.25){}; \node[label=right:y](R)at(1.1,1.25){};\draw[line width=.05em](L)--(R);
\node[A](L)at(-1,1.25){}; \node[label=left:z](R)at(-2.1,1.25){};\draw[line width=.05em](L)--(R);
\node[label=left:(IV)]()at (-2.5,.75){};
\end{tikzpicture}

\begin{tikzpicture}[scale=.8]
  \foreach \i in {-1,...,0}{
   \draw[line width=.05em](\i,0.75)--(\i,2);}    
\node[label=above:{$r_x=r_y$}]()at(-1,0){};

\node[A](L)at(-1,1.6){}; \node[label=right:w](R)at(0.6,1.6){};\draw[line width=.05em](L)--(R);
\node[A](L)at(-1,1.25){};\node[label=left:x](R)at(-2.3,1.25){};\draw[line width=.05em](L)--(R);
\node[A](L)at(-1,.9){};\node[label=left:y](R)at(-2.3,.9){};\draw[line width=.05em](L)--(R);
\node[A](L)at(0,1.25){};\node[label=right:z](R)at(.73,1.25){};\draw[line width=.05em](L)--(R);
\node[label=left:(II)]()at (2,.6){};
\end{tikzpicture}\hskip1.8cm
\begin{tikzpicture}[scale=.8]
  \foreach \i in {-1,...,0}{
   \draw[line width=.05em](\i,0.25)--(\i,2);}    
\node[label=above:{$r_x=r_y$}]()at(0,-0.35){};

\node[A](L)at(0,1.6){}; \node[label=right:w](R)at(1.1,1.6){};\draw[line width=.05em](L)--(R);
\node[A](L)at(0,1.25){};\node[label=left:x](R)at(-2.3,1.25){};\draw[line width=.05em](L)--(R);
\node[A](L)at(0,.9){};\node[label=left:y](R)at(-2.3,.9){};\draw[line width=.05em](L)--(R);
\node[A](L)at(0,.55){};\node[label=right:z](R)at(.73,.55){};\draw[line width=.05em](L)--(R);
\node[label=left:(III)]()at (-2.4,.25){};
\end{tikzpicture}
\end{center}
\caption{Diamonds in the canonical representation.}
\end{figure}
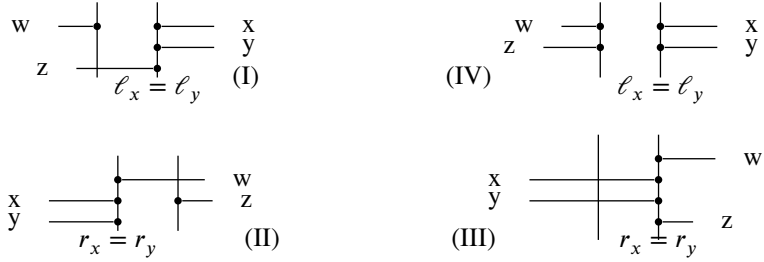
For (II) and (III),   $\rho_w\red \rho_x$ and $\rho_w\red \rho_y$ are arcs of $G_P$, we have
	$r_x = \ell_w = r_y$.  
	
		If   $\rho_x\blue \rho_z$ is in $G_P$, then $\ell_z = r_x+1= r_y+1$ which implies $\rho_y\blue\rho_z$ is  in $G_P$.
	
	If   $\rho_z\red \rho_x$ is in $G_P$, then $\ell_z = r_x=r_y$ which implies $\rho_z\red \rho_y$ is in $G_P$.
\end{proof}

Lemma \ref{diamond} remains true if the vertices of the diamond are not distinct. Actually, the endpoints of any red arc can be identified making the arc a red loop. It is obviously cannot work for blue arcs, since $G_P$ has no blue loops. 
Notice, that according to 
Lemma \ref{diamond},  if  three among the four colored arcs/loops  of a diamond  are known, then the fourth arc/loop is in $G_P$, and its color is determined. 
		 	
\begin{lemma} \label{multipartition} 
If $L_1$ and  $L_2$ are color-alternating directed paths in the key graph from vertex $\rho_a$ to vertex $\rho_b$, and they start with arcs of the same color, then $L_1$ and $L_2$ have the same length and they enter $\rho_b$ with arcs of the same color.
\end{lemma}
\begin{proof}  
The proof is induction on the length $h=|L_1|>2$.  Let
$L_1=(\rho_a,\rho_x,\rho_y,\ldots,\rho_b)$ and $L_2=(\rho_a,\rho_{u},\rho_v,\ldots,\rho_b)$.
By assumption, the arcs $\rho_{u}\rightarrow\rho_v$ and $\rho_{u}\rightarrow\rho_y$ have the same color.  So,
by Lemma 4.1 (I) or (II), the claim follows for $h=3$.

For $h>3$, consider the paths $L_1^\prime=(\rho_u,\rho_v,\ldots,\rho_b)$ and $L_2^\prime=(\rho_u,\rho_y,\ldots,\rho_b)$.
These color-alternating directed paths from $\rho_u$ to $\rho_b$ start with arcs of the same color, as we saw when $h=3$, and $|L^\prime_1|=h-1$.
By induction,  the paths $L^\prime_1$ and $L^\prime_2$, hence  $L_1$ and $L_2$ have the same length, and they enter $\rho_b$ with arcs of the same color.
\end{proof}

\subsection{Circulations}
\label{CirculationsSubSection}
To identify the Schrijver system for the length polyhedron, we begin with the system of cycle inequalities derived from cycles
of the {\it \FaceGraph}
(which is TDI, by Proposition \ref{CYCTDI})
and iteratively discard inequalities that are a non-negative integral
linear combination of smaller cycle inequalities (`smaller' here refers to the weak order defined
in (\ref{WeakOrderOnInequalities})).  The process of detecting inequalities that must be discarded begins by
NOT discarding all inequalities that correspond to loops of the {\it \FaceGraph}.  These {\em loop inequalities} have the
form $\left\langle 0\leq \rho_x\right\rangle$, for $x \in X$ such that $x$ is assigned an interval of length zero in the 
canonical representation.  These inequalities are clearly not non-negative linear combinations of other cycle inequalities, so
they will form part of the Schrijver system.

By definition, a cycle inequality $W$ is {\em redundant} if it is a non-negative linear combination of other cycle inequalities:
\begin{equation}\label{redundantW}
W=\sum\limits_{i=1}^t\alpha_i W_i,
\end{equation}
where  $W_i$ is a cycle inequality not equivalent to $W$,  and $\alpha_i\geq 0$  is a rational coefficient, for $i=1,\ldots,t$.
Our goal is to prove that a redundant cycle inequality is a non-negative {\em integral} 
linear combination of smaller cycle inequalities (Theorem \ref{integerNonNegativeLinearCombinationForInequalities}).
Thus, to decide whether a cycle inequality should be discarded, it suffices to verify the feasibility of a linear program
(which can be done in polynomial-time provided the number of smaller inequalities is polynomially bounded).
In order to prove Theorem \ref{integerNonNegativeLinearCombinationForInequalities}, 
we employ circulations of the \FaceGraph, which we now introduce.

Given a directed graph $D=(V,A)$ and a vertex $v \in V$,  the set of arcs {\em entering} $v$ is denoted by $\delta^-(v) = \{a \in A :  a = (u,v) \mbox{ for some } u \in V\}$, and  the set of arcs {\em exiting} $v$ is denoted by 
 $\delta^+(v) = \{a \in A : a = (v,u) \mbox{ for some } u \in V\}$.
A {\em circulation} on $D$ is a function $f : A \rightarrow \mathbb{R}$ such that $f(a) \geq 0$, for all $a \in A$, and
$\sum\limits_{a \in \delta^-(v)} f(a) = \sum\limits_{a \in \delta^+(v)} f(a)$, for all $v \in V$.

Observe that if $f$ is a circulation on digraph $D$, then $cf$, the scaling of all values of $f$ by a positive constant $c$, is also a circulation.
The value $f(a)$ is called the {\em flow} along the arc $a$.
A circulation $f$ is {\em an integral circulation} if $f(a)$ is integer, for all arcs $a$. 
The {\em zero circulation} is the circulation that assigns value zero to all arcs. 
The {\em support} of a circulation $f$, denoted $supp(f)$, is the subdigraph  
 induced by the set of all arcs with non-zero flow, where it is understood that
isolated vertices are eliminated if they appear.

The
support of a circulation defined on the  \FaceGraph $G_P=(V,A)$ is a subdigraph inheriting all of the arc colors and arc weights.  If $f$ is a circulation of the \FaceGraph, define its weight to be
$$W(f) = \sum_{a \in A} f(a)w(a).$$

It is understood here that the arc weight inequalities can be added in the usual way with appropriate cancellations and simplifications in the sense that  
$\left\langle a \leq b\right\rangle$ plus $\left\langle c \leq d\right\rangle $ sums to $\left\langle a + c \leq b + d\right\rangle $.
A {\em zero weight circulation} is one whose weight reduces to $\left\langle 0 \leq 0\right\rangle $.  
Two circulations are {\em equivalent} if they have the same weight.

If $C$ is a cycle of $G_P$, we identify it with the circulation assigning value $1$ to all arcs of $C$
and zero to all other arcs of $G_P$; this is a {\em cycle circulation}.
The cycle inequality for the length polyhedron $Q_P$ determined (via Fourier-Motzkin elimination) by cycle $C$ is precisely the weight of this cycle circulation, which we denote $W(C)$.
Vertices of $supp(f)$ that correspond to variables in $W(f)$ are {\em basic}; other vertices are {\em non-basic}.

It is well known (see, for example, p.135 of \cite{ConfortiCornuejolsZambelli}) that every  (integral) circulation 
is a non-negative (integral) linear combination of cycle circulations.  Here we quickly review this straightforward procedure that we call {\it rational decomposition}.

\begin{proposition}
\label{real_cycle_decomp}
If a circulation $f$ is not a single cycle, then it has a decomposition of the form $f = \sum_{i=1}^t \alpha_i C_i,$ where $\alpha_i>0$,  $C_i $ is a cycle in $supp(f)$, for $i=1,\ldots,t$, and $t\geq 2$. 
\end{proposition}
\begin{proof}
 Consider a circulation $f$ and a directed cycle $C$ in $supp(f)$.  Define $$f_{\min}(C) := \min\{f(a): a \mbox{ is an arc of } C\}.$$
The {reduction of $f$ by $C$}, denoted $f_C$, is the circulation defined as
$$ 
f_C(a) := \left\{
\begin{array}{ll}
	f(a) & \mbox{ if $a$ is not in $C$ } \\
	f(a) - f_{\min}(C) & \mbox{ if $a$ is in $C$. }
\end{array} \right. 
$$
Circulation $f_C$ is the result of reducing $f$ by $C$.
Note that $W(f_C) = W(f) - f_{\min}(C) W(C)$.  Also note that $supp(f_C)$ has fewer arcs than $supp(f)$ since at least
one arc of $C$ is assigned zero flow by $f_C$.   

Now $f_C$ could be further reduced using any cycle in $supp(f_C)$.
In this way, repeated reductions decompose the original circulation $f$ into a finite, non-negative linear combination of cycle circulations.
Indeed, this shows that there exists a positive integer $t$, cycles $C_1,\ldots,C_t$, and positive scalars $\alpha_1,\ldots,\alpha_t$ such that $f$ takes the form
\begin{eqnarray} \label{linearCombo}
f = \sum_{i=1}^t \alpha_i C_i,
\end{eqnarray}
called a {decomposition} of $f$,
where $\alpha_i$ is equal to the minimum flow among arcs of a cycle $C_i$ in the support of the circulation that results from
reducing $f$ by $C_1,\ldots,C_{i-1}$.
\end{proof}

Observe that if $f$ is an integral circulation in Proposition \ref{real_cycle_decomp}, then all $\alpha_i$'s are integral; in this case,
$f$ decomposes into cycles, that is, $supp(f)$ is a non-negative integral linear combination of cycle circulations. The rational decomposition (\ref{linearCombo}) of a given circulation $f$ is not unique;  indeed, there may be many choices at each reduction step to choose a cycle
to perform the reduction.  
 Each rational decomposition of a circulation $f$ 
 corresponds to a non-negative linear combination of cycle inequalities:
\begin{eqnarray} \label{InequalityLinearCombo}
	W(f) = \sum_{i=1}^t \alpha_i W(C_i).
\end{eqnarray} 
Therefore, inequality $W(f)$  is valid for the length polyhedron $Q_P$ as
the non-negative sum of cycle inequalities defining $Q_P$.

Notice that each cycle inequality in (\ref{CYC}) contains a variable on the right-hand side, therefore, 
the only circulation that has the zero-weight inequality $\left\langle 0\leq 0 \right\rangle$ is the zero circulation, that is, a circulation having zero flow on each arc.  The total dual integrality derived in Corollary \ref{TDI} has an important consequence on the decomposition of a circulation 
into cycles.

\begin{proposition} \label{Circulation2LinearCombinationOfCycleInequalities}
   Suppose that $P$ is an interval order and $f$ is a circulation of its \FaceGraph.
   If the weight of $f$ is a cycle inequality, then $W(f)$ is a non-negative integer linear combination of cycle inequalities determined by cycles from $supp(f)$.
\end{proposition}  
\begin{proof}  Cycle inequalities have the form shown in (\ref{CYC}).  
	Let $\boldsymbol{z} \boldsymbol{\rho} \leq -\gamma$ be the cycle inequality $W(f)$ defined by the circulation $f$, where
	$\gamma$ is a non-negative integer, $\boldsymbol{z}$ is a $\{0,\pm 1\}$-vector of dimension $n$, and  $\boldsymbol{\rho} = (\rho_1,\ldots,\rho_n)^T$.
	Recast this inequality for variables $\boldsymbol{x} = (\ell_1,\ldots,\ell_n,\rho_1,\ldots,\rho_n)^T$ writing it as $\boldsymbol{c} \boldsymbol{x} \leq -\gamma$,
	where it is understood that $\boldsymbol{c} = (0,\ldots,0,z_1,\ldots,z_n)$ is $2n$-dimensional after padding by $n$ leading zeros.
	
	Consider $\LocationAndLength \boldsymbol{x} \leq {\boldsymbol{b}}_P$ defined by (\ref{LRO2}), and let $M_f \boldsymbol{x} \leq {\boldsymbol{b}}_f$ be the subsystem 
		 that is determined 
	by picking out the rows of $\LocationAndLength$  
	and ${\boldsymbol{b}}_P$ corresponding to arcs of the \FaceGraph in  $supp(f)$.
	Corollary \ref{TDI} implies that the system $M_f \boldsymbol{x} \leq {\boldsymbol{b}}_f$ is TDI.
	
	The linear programming duality equation
	$$\max\{\boldsymbol{c}\boldsymbol{x} : M_f\boldsymbol{x} \leq \boldsymbol{b}_f \} = \min\{\boldsymbol{y}\boldsymbol{b}_f  : \boldsymbol{y}\geq \boldsymbol{0}, \boldsymbol{y}M_f  = \boldsymbol{c}\}$$
	has value $-\gamma$, by assumption.  Indeed, the circulation $f$ witnesses that the maximum and minimum have solutions yielding $-\gamma$.  The total dual integrality guarantees that there exists a non-negative integral optimal solution $y$ to the minimum.  This corresponds to an integral circulation of $supp(f)$, call it $g$, such that $W(g)=W(f)$.
	Because every integral circulation 
	is a non-negative integral linear combination of cycle circulations, the theorem follows.
\end{proof}

The  integral decomposition obtained via  
Proposition \ref{Circulation2LinearCombinationOfCycleInequalities} might seem to establish our main goal: that
any redundant cycle inequality is a non-negative integral linear combination of other cycle inequalities.
However, this is not entirely accurate, as $supp(f)$ may include cycles that themselves realize the cycle inequality $W(f)$.  In such cases, 
Proposition \ref{Circulation2LinearCombinationOfCycleInequalities} would merely show that a cycle inequality can be expressed 
as an integral linear combination of itself - a triviality.  To overcome this impasse, we need to remove these problematic cycles by appropriately reducing the circulation $f$, a task achieved 
by Proposition \ref{sublimation_step} within the Sublimation Loop.  We tackle this task in the next section.  

\section{Computing the Schrijver system}
\label{mainproof}
Because cycle inequalities correspond to supporting hyperplanes of the length polyhedron (they are tight for the canonical representation) and 
the system of cycle inequalities is TDI (Proposition \ref{CYCTDI}),
the Schrijver system consists of cycle inequalities that are
not the non-negative integral combination of other cycle inequalities.

The algorithm to construct the Schrijver system begins by listing all cycle inequalities (in order according to the weak order defined in Section \ref{weak}); this list is referred as to the {\it Weak-list}.
From this Weak-list, the appropriate inequalities are identified and incorporated 
into another list, called the {\it Schrijver list}, which eventually contains the minimal TDI system for the length polyhedron. 

The Schrijver list is initialized with the loop inequalities as mentioned at the beginning of subsection \ref{CirculationsSubSection}.  
We show that the  Schrijver list, initialized with these inequalities and then extended with 
all cycle inequalities that are not expressible as non-negative rational linear combinations of smaller cycle inequalities (in the weak order), 
constitutes the Schrijver system of the length polyhedron $Q_P$. According to Theorem \ref{WeakOrderSignificance}, these are precisely the inequalities that should be included.

When the next inequality $W$ of the Weak list is investigated, we may find that it is not a rational, non-negative linear combination of other cycle inequalities.
In this case, $W$ is clearly not a non-negative integral combination of other cycle inequalities, so it is added to the Schrijver list.
On the other hand, if $W$ is a rational, non-negative linear combination of other inequalities, the main theorem of this section (Theorem \ref{integerNonNegativeLinearCombinationForInequalities}) 
implies that $W$ is necessarily
a non-negative integral combination of other cycle inequalities; therefore, it is not added to the Shrijver list, in this case.  
This is the main theorem of this section:
\begin{theorem} \label{integerNonNegativeLinearCombinationForInequalities}
	Suppose that $P$ is an interval order. 
	Each redundant cycle inequality of its length polyhedron can be expressed 
	as a non-negative integral linear combination of other cycle inequalities.
\end{theorem}

To prove Theorem \ref{integerNonNegativeLinearCombinationForInequalities}, we begin with
an inequality that is a rational, non-negative rational linear combination of other cycle inequalities and we must 
find an equivalent non-negative integral combination of other cycle inequalities.
The algorithm to find this integral combination works with circulations of the \FaceGraph
and involves two loops.  The flow chart of the first loop that we call the Sublimation Loop is shown in Figure \ref{SubLoop}; 
the second one is the $2$-cycle Loop whose flow chart is presented in Figure \ref{2-cycle_Loop}.
The loops of the algorithm use basic properties of the key graph we introduce in the next subsection.
The proof of Theorem 5.1 begins in subsection \ref{SublimationLoopSubSection}.

\subsection{Concordant cycles and their properties} 
The structure of a \FaceGraph implies remarkable properties of its directed cycles.  In this subsection
we develop the notion of concordant cycles and their properties.
Let $f$ be a circulation of a \FaceGraph.
A cycle $C$ in $supp(f)$ is {\em concordant} (with $f$) if  $C$, interpreted itself as a cycle circulation, and the circulation $f$ are equivalent, that is, $W(C) = W(f)$. 
If $W(C)\neq W(f)$ then $C$ is   {\em discordant}. 
A vertex of $supp(f)$ is {\em mixed} if either all arcs entering it do not all have the same color 
or all arcs leaving the vertex do not have the same color. 

\begin{lemma}\label{no_mixed} Suppose $f$ is a circulation of a \FaceGraph.  If
	$C_1, C_2$ are concordant cycles in $supp(f)$,
	then $C_1\cup C_2$ does not contain a mixed vertex.
\end{lemma}
\begin{proof}  Consider an arbitrary, hypothetical mixed vertex $\rho_x$ in $C_1\cup C_2$.
	There are two cases:   
	($\rho_x\red \rho_u$ is in $C_1$ and $\rho_x\blue \rho_v$ is in $C_2$) or 
	($\rho_u\red \rho_x$ is in $C_1$ and $\rho_v\blue \rho_x$ is in $C_2$ ).
	\begin{figure}[htp]
		\begin{center}
			\begin{tikzpicture}[scale=.8, ->,>=stealth',auto,node distance=3cm,
				thick,main node/.style={circle,draw,font=\sffamily\small\bfseries}, node/.style={}]
				
				\node[A,label=above:$\rho_x$] at (0,.75) (1) {};
				\node[A,label=above:$\rho_u$] at (1.25,1.5) (2) {};
				\node[A,label=below:$\rho_v$] at (1.25,0) (4) {};
				\node[label=right:$$] at (-1.25,1.5) (3) {};
				\node[label=right:$$] at (-1.25,0) (5) {};
				
				\path[draw, thick, every node/.style={font=\sffamily\small}]
				(1) edge [red,line width=.08em] node [] {} (2)
				(1) edge [cyan,double,line width=.08em] node [] {} (4)
				(3) edge [red,double,line width=.08em] node [] {} (1)
				(5) edge  [cyan,line width=.08em]  node [] {} (1);
			\end{tikzpicture}\hskip1.52cm
			\begin{tikzpicture}[scale=.8, ->,>=stealth',auto,node distance=3cm,
				thick,main node/.style={circle,draw,font=\sffamily\small\bfseries}, node/.style={}]
				
				\node[A,label=left:$\rho_b$] at (0,.75) (1) {};
				\node[A,label=above:$$] at (1.25,1.5) (2) {};
				\node[A,label=below:$$] at (1.25,0) (4) {};
				\node[A,label=right:$\rho_w$] at (6.25,.75) (3) {};
				\node[A,label=right:$$] at (5,1.5) (6) {};
				\node[A,label=right:$$] at (5,0) (5) {};
				\node[A,label=right:$$] at (3.75,1) (21) {};
				\node[A,label=right:$$] at (2.25,1) (10) {};
				\node[A,label=right:$$] at (2.5,0) (20) {};
				\node[A,label=right:$$] at (4,0) (11) {};
				
				\path[draw, thick, every node/.style={font=\sffamily\small}]
				(1) edge [cyan] node [] {} (2)
				(1) edge [cyan,double,line width=.05em] node [] {} (4)
				(6) edge [red,double,line width=.05em] node [] {} (3)
				(5) edge  [cyan]  node [] {} (3)
				(21) edge [cyan,double,line width=.05em] node [] {} (6)
				(2) edge [red,line width=.05em] node [] {} (10)
				(4) edge [red,double,line width=.05em] node [] {} (20)
				(10) edge  [dashed,line width=.05em]  node [] {} (11)
				(11) edge  [red,line width=.05em]  node [] {} (5)
				(20) edge  [double,dashed,line width=.05em]  node [] {} (21);	
				\node[label=above:$L_1$] at (1.8,1.25) () {};				
				\node[label=above:$L_2$] at (4.25,1.25) () {};	
			\end{tikzpicture}
		\end{center}
		\caption{No mixed vertex in the union of two concordant cycles.}
		\label{NoMixedVertexFigure}
	\end{figure}
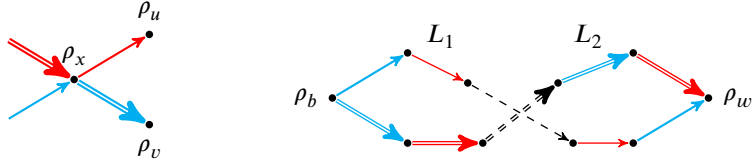
	Because both $C_1$ and $C_2$ are concordant,
	either case implies that $\rho_x$ is a non-basic vertex.
	Consequently, incoming arcs to a mixed vertex must have opposite color (see left image of Figure \ref{NoMixedVertexFigure}). 
	
	Now assume, to the contrary, that there is a mixed vertex in $C_1\cup C_2$.  Among all pairs $(\rho_b,\rho_w)$ such that $\rho_b$ is basic and $\rho_w$ is mixed, choose
	a pair minimizing the length of a shortest directed path in $C_1$ or $C_2$ from $\rho_b$ to $\rho_w$ in $C_1\cup C_2$.
	For $i=1,2$, let $L_i\subset C_i$ be the path from  $\rho_b$ to $\rho_w$
	(see right image of Figure \ref{NoMixedVertexFigure}).
	First observe that $L_1$ and $L_2$ are internally vertex-disjoint by the choice of the pair $(\rho_b,\rho_w)$.
	Therefore, $L_1, L_2$ are color-alternating directed paths.
	Furthermore, $L_1, L_2$ start with the same color, because $\rho_b$ is a basic vertex.  
	By Lemma \ref{multipartition}, the color of the arc of $L_1$ entering $\rho_w$ is the same as the color of the arc of $L_2$ entering $\rho_w$, contradicting that $\rho_w$ is a mixed vertex.
\end{proof}


\begin{proposition} \label{no_discordant_cycle}
	Let  $C_i,\ldots, C_t$ ($t\geq 2$) be concordant cycles in $supp(f)$, for some circulation $f$. If 
	$\bigcup_{i=1}^t C_i$ contains no discordant cycle, then 
	all cycles $C_i$, $i=1,\ldots,t$, visit the basic vertices in $supp(f)$ in the same circular order.
\end{proposition}

\begin{proof} The claim is trivial if $supp(f)$ contains at most two basic vertices, so we presume there are at least three basic vertices.
	Assume, on the contrary, that there are basic vertices $\rho_{a}, \rho_{b}, \rho_{c}$  and two concordant cycles $C_i$ and $C_j$ such that
	$C_i$ visits these basic variables in order $\rho_{a}, \rho_{b}, \rho_{c}$, and $C_j$ visits them in order
	$\rho_{a}, \rho_{c}, \rho_{b}$.  Consider the digraph obtained from $C_i \cup C_j$ by deleting vertex $\rho_b$.  This digraph
	contains a path from $\rho_{a}$ to $\rho_{c}$ (from $C_j$) and a path from $\rho_{c}$ to $\rho_{a}$ (from $C_i$).  Consequently,
	it contains a directed cycle.  This cycle exists in $\bigcup_{i=1}^t C_i$ and it must be discordant because it avoids basic vertex $\rho_b$, a contradiction.
\end{proof}
	
The next concordant cycle property is an anti-Helly property that helps restrict the decomposition of the
cycle analyzed in the Sublimation Loop into the linear combination of just two cycles. 

\begin{proposition}\label{helly}  Let $g$ be a circulation with concordant cycles
	 $C_1,\ldots, C_t$ ($t\geq 2$).  If $C_i\cup C_j$ contains no discordant cycle, for all $1\leq i<j\leq t$,  then   $\bigcup_{i=1}^t C_i$ has no  discordant cycle.
\end{proposition}
\begin{proof}
	Applying Proposition \ref{no_discordant_cycle} to every pair $C_1$, $C_i$, $i=2,\ldots, t$, we find that all $C_i$'s visit basic vertices in the same circular order. The union of the directed paths between any two consecutive basic vertices along these cycles are color-alternating as described by Lemma \ref{multipartition}.  Because there are no mixed vertices (Lemma \ref{no_mixed}),  variables that appear in discordant's cycle inequality
	are a proper subset of the variables appearing in the concordant cycle inequality.  
	So, the only way that $\bigcup_{i=1}^t C_i$ has a discordant cycle is if some pair of concordant cycles, say $C_1$ and $C_2$, share
	a non-basic variable $\rho_x$ that exists between different pairs of consecutive basic variables of these cycles.
	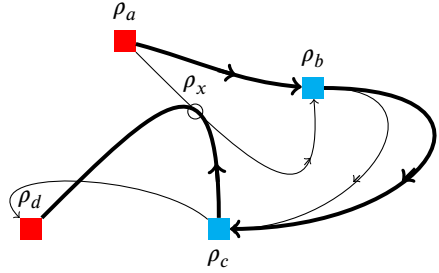
\begin{figure}[htp]
		\begin{center}
			\begin{tikzpicture}[scale=1.25]
				
				\node[Yr,label=above:$\rho_d$](x) at (0,0){};
				\node[Yb,label=below:$\rho_c$](y) at (2,0){};
				\draw[double,line width=.15em] (x) edge[out=45, in=90, looseness=2.5] (y);
				\draw[->] (y)  edge[out=-18 in=-180, looseness=1](x);
				\node[Yr,label=above:$\rho_a$](A) at (1,2){};
				\node[Yb,label=above:$\rho_b$](B) at (3,1.5){};
				\draw[->,double,line width=.15em] (A) edge[out=-15, in=175, looseness=1] (B);
				\draw[->] (A) edge[out=-45, in=275, looseness=2] (B);
				\draw[->] (B) edge[out=0, in=0, looseness=2] (y);
				\draw[->,double,line width=.15em] (B) edge[out=0, in=0, looseness=3] (y);
				\node[X,label=above:$\rho_x$]() at (1.75,1.25){};
				\draw[<-, line width=.15em] (3.95,.5) -- ++(45:0.2); 
				\draw[->, line width=.15em] (1.99,.5) -- ++(95:0.2); 
				\draw[<-, line width=.05em] (3.43,.5) -- ++(45:0.1); 
				\draw[->, line width=.04em] (2.91,.68) -- ++(50:0.1); 
				\draw[->, line width=.15em] (1.99,1.69) -- ++(-15:0.2);         
			\end{tikzpicture} 
		\end{center}
		\caption{If there is no discordant cycle in the union of two concordant cycles, then the union do not share non-basic vertices.}
		\label{DoNotShareBasic}
	\end{figure}
	Let $\rho_x$ lie between basic vertices $\rho_a$, $\rho_b$ on $C_1$, and between basic vertices $\rho_c$, $\rho_d$ on $C_2$ as seen in Figure \ref{DoNotShareBasic} ($\rho_b=\rho_c$ or $\rho_d=\rho_a$  is possible).
	The pair $C_1, C_2$ violates the hypothesis because the path from $\rho_c$ to $\rho_x$ along $C_2$, extended along $C_1$ to $\rho_b$, and then further extended along any path from $\rho_b$ to $\rho_c$,
	forms a digraph with a cycle that is discordant due to avoiding $\rho_a$.  
\end{proof}	

Contrapositive form of Proposition \ref{helly}  is used to prepare the input for the $2$-cycle Loop.
\begin{corollary} Let $f=\sum_{i=1}^t \alpha_i C_i$ be a circulation, where $t\geq 2$,  $C_1,\ldots, C_t$  are concordant cycles, and $supp(f)$ contains a discordant cycle. 
	There exists a circulation $g$ equivalent to $f$ such that $g=\frac12 C_i+\frac12 C_j$, for some $C_i$ and $C_j$, $1\leq i< j\leq t$ and $supp (g)$ contains a discordant cycle.
	\label{2cycles}
\end{corollary}

\subsection{The Sublimation Loop} 
\label{SublimationLoopSubSection}

 Suppose that $W$ is a redundant cycle inequality, say $W=\sum_{i=1}^r \alpha_i W_i$ expresses $W$ as a non-negative rational linear combination of other cycle inequalities,
where $\alpha_i \in \mathbb{Q}^+$ and $r \geq 2$.
The algorithm begins by constructing a circulation $f$ of the \FaceGraph.
To construct this circulation, a cycle of the \FaceGraph representing the inequality $W_i$ is chosen for each $i=1,\ldots,r$.
The arcs of the circulation $f$ are the arcs appearing in these chosen representative cycles.
The flow of an arc is defined as the sum of the coefficients belonging to the cycles containing that arc. 
Observe that $W(f)=W$.  It is also important to note that, because $W$ is redundant, $supp(f)$ contains a discordant cycle. 

The crucial step in the proof of Theorem \ref{integerNonNegativeLinearCombinationForInequalities} involves eliminating concordant 
cycles from $supp(f)$, resulting in a non-trivial integral decomposition of $W $ in terms of discordant cycle inequalities.
This step 
is realized  in the next proposition  by  the {\it sublimation} of $f$.

\begin{proposition} \label{sublimation_step}
	If $f$ is a circulation  and   $supp(f)$ contains at least one discordant and one concordant cycle, then there exists a circulation $h$ equivalent to $f$ such that $supp(h)$ contains a discordant cycle.
	Moreover, $supp(h)$ has fewer arcs than $supp(f)$. 	
\end{proposition} 
\begin{proof}  Let $D$ and $C$ be a discordant and a concordant cycle in $supp(f)$, respectively.  Let $f = \beta D+\alpha C+ \sum_{i=1}^r\alpha_i C_i$ be a rational decomposition of $f$ as described in Proposition \ref{real_cycle_decomp} starting with $D$ and $C$.  Consider the circulation $g$ defined by
	$$ g(a) := \left\{
	\begin{array}{ll}
		f(a) & \mbox{ if $a$ is not in $C$ } \\
		f(a) - \alpha & \mbox{ if $a$ is in $C$. }
	\end{array} \right. $$
	Clearly $g = \beta D+ \sum_{i=1}^r\alpha_i C_i$ is a decomposition of $g$ and $W(g)=W(f)(1 - \alpha)$.
	Observe that $0 < \alpha< 1$, since $g$ is not the zero circulation ($supp(g)$ contains $D$) and its weight cannot be
	a negative sum of cycle inequalities.  Consequently, the circulation $h = \frac{1}{1 - \alpha} g$ is equivalent to $f$; 
	together with $D$ every  discordant
	cycle  in $supp(h)$ remains discordant, and  $supp(h)$ has fewer arcs than $supp(f)$.
\end{proof}

\begin{center}
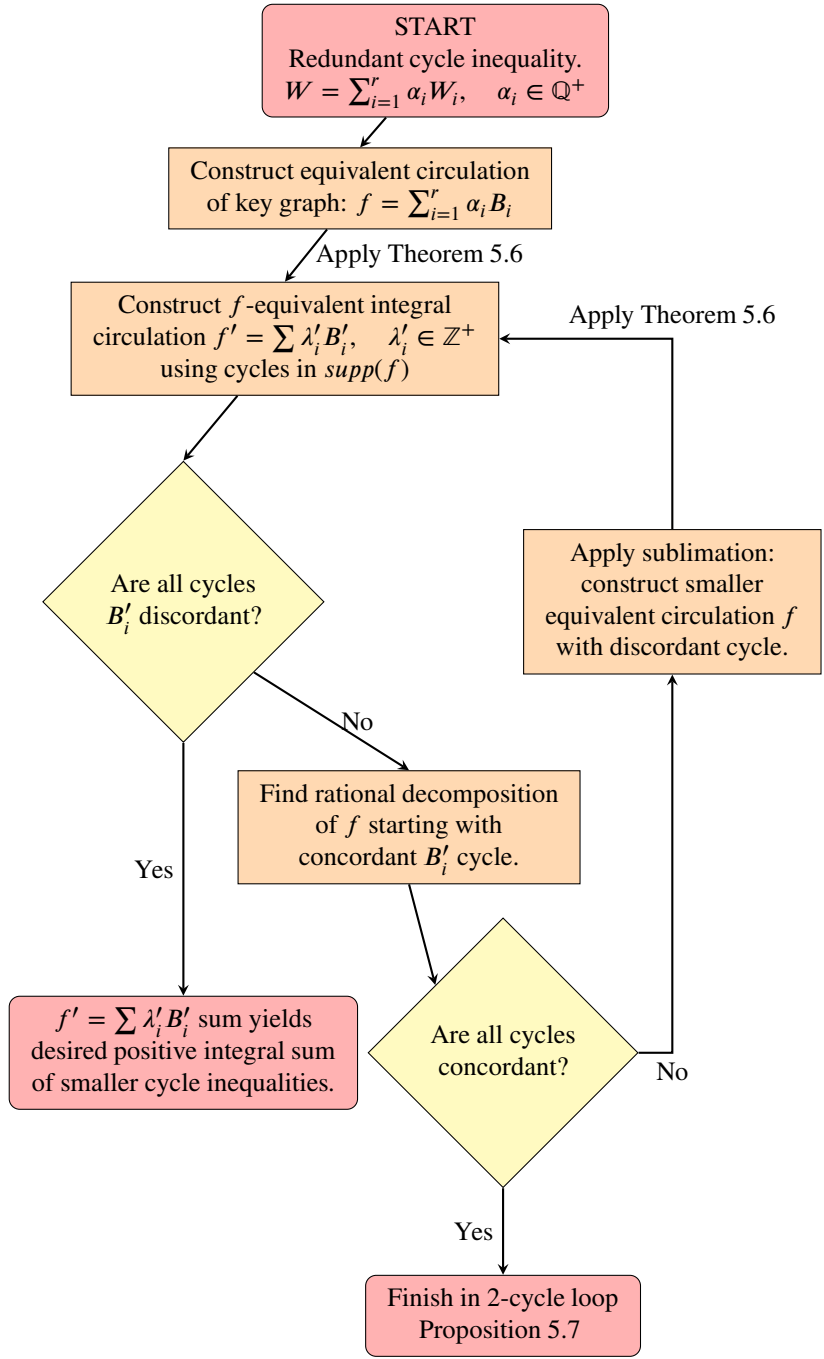
\begin{figure}[htp]
	\thisfloatpagestyle{empty}
	\begin{center}
	\begin{tikzpicture}[scale=0.25]
	\node (start) [startstop] {$\begin{array}{c} \mbox{START} \\ \mbox{Redundant cycle inequality.} \\ W = \sum_{i=1}^r \alpha_i W_i, \quad \alpha_i \in \mathbb{Q}^+ \end{array}$};
	\node (CreateEquivalentCirculation) [process, below of=start, xshift=-1cm, yshift=-0.7cm] {$\begin{array}{c} \mbox{Construct equivalent circulation} \\ \mbox{of key graph: } f = \sum_{i=1}^r \alpha_i B_i \end{array}$};
	\node (ConstructEquivIntegralCirculation) [process, below of=CreateEquivalentCirculation, xshift=-1cm, yshift=-1 cm] {$\begin{array}{c} \mbox{Construct }f\mbox{-equivalent integral} \\ \mbox{circulation } f^\prime = \sum \lambda^\prime_i B^\prime_i, \quad \lambda^\prime_i \in \mathbb{Z}^+ \\ \mbox{using cycles in } supp(f) \end{array}$};	
	\node (AllDiscordant) [decision, below of=ConstructEquivIntegralCirculation, xshift=-1.35 cm, yshift=-2.5cm] {$\begin{array}{c} \mbox{Are all cycles} \\ B^\prime_i \mbox{ discordant?}\end{array}$};
	\node (finish1) [startstop, below of=AllDiscordant, xshift=-0 cm, yshift=-5cm] {$\begin{array}{c} f^\prime = \sum \lambda^\prime_i B^\prime_i \mbox{ sum yields} \\ \mbox{desired positive integral sum} \\ \mbox{of smaller cycle inequalities}. \end{array}$};
	\node (RationalDecompositionWithConcordant) [process, below of=AllDiscordant, xshift=3cm, yshift=-2cm] {$\begin{array}{c} \mbox{Find rational decomposition} \\ \mbox{of } f \mbox{ starting with} \\ \mbox{concordant } B^\prime_i \mbox{ cycle.} \end{array}$};

	\node (AllConcordant) [decision, below of=RationalDecompositionWithConcordant, xshift=1.25 cm , yshift=-2 cm] {$\begin{array}{c} \mbox{Are all cycles} \\ \mbox{concordant?} \end{array}$};
	\node (stop) [startstop, below of=AllConcordant, xshift=0 cm , yshift=-2.5 cm] {{$\begin{array}{c} \mbox{Finish in $2$-cycle loop} \\ \mbox{Proposition \ref{Case_d} } \end{array}$}};
	
	\node (sublimation) [process, right of=AllDiscordant, xshift=5.5cm, yshift=0cm] {{$\begin{array}{c} \mbox{Apply sublimation:} \\ \mbox{construct smaller} \\ \mbox{equivalent circulation } f \\ \mbox{with discordant cycle.} \end{array}$}};
	
	\draw [arrow] (start) -- (CreateEquivalentCirculation.north) ;
	\draw [arrow] (CreateEquivalentCirculation) -- node[anchor=west] {Apply Theorem \ref{sublimation_step}}  (ConstructEquivIntegralCirculation.north);
	\draw [arrow] (AllDiscordant.south) -- node[anchor=east] {Yes} (finish1);
	\draw [arrow] (AllDiscordant.south east) -- node[anchor=west] {No} (RationalDecompositionWithConcordant.north);
	\draw [arrow] (ConstructEquivIntegralCirculation) -- (AllDiscordant.north);
	\draw [arrow] (AllConcordant) -- node[anchor=east] {Yes} (stop);
	\draw [arrow] (RationalDecompositionWithConcordant.south) -- (AllConcordant.north west);
	\draw [arrow] (AllConcordant.east) -| node[anchor=north] {No} (sublimation.south);
	\draw [arrow] (sublimation) |- node[anchor=south] {Apply Theorem \ref{sublimation_step}} (ConstructEquivIntegralCirculation);
	
	\end{tikzpicture}
	\end{center}

	\caption{A flowchart depicting part of the algorithm that takes a cycle inequality that is a positive linear combination of other cycle inequalities and returns an equivalent integral sum of smaller cycle inequalities.  This chart highlights the sublimation loop.}
\label{SubLoop}
\end{figure}
\end{center}


 The Sublimation Loop either produces the required integral linear combination of the redundant cycle inequality $W$ or it  exits with an equivalent circulation  $g=\frac12 C_1 +\frac12 C_2$, where $C_1$, $C_2$ are concordant cycles and $supp(g)$ contains a discordant cycle.  Dealing with this circulation is the task of the $2$-cycle Loop that follows the Sublimation Loop.

The Sublimation Loop starts with the application of Proposition \ref{Circulation2LinearCombinationOfCycleInequalities} to $f$.  
This produces an integral decomposition $f=\sum_{i=1}^t \lambda_i C_i$, where $\lambda_i>0$ is integer,  and $C_i$ is a cycle in $supp(f)$,
for $i=1,\ldots,t$. If $C_i$ is discordant for every  $1\leq i\leq t$, then the procedure terminates with a required integral combination of $W$. Otherwise, the decomposition of $f$ given by
Proposition \ref{Circulation2LinearCombinationOfCycleInequalities} is trivial ($t = 1$) and we need to find another decomposition.  Note that in this
latter trivial case, a concordant cycle is returned and, from this concordant cycle, a rational decomposition can be started.
Starting with this concordant cycle, Proposition \ref{real_cycle_decomp} returns a rational decomposition containing at least two cycles (because $f$ is not a single cycle circulation.  Indeed, the
decomposition begins with a concordant cycle and the circulation contains a discordant cycle).

If the decomposition is {\em heterogeneous}, that is, among the cycles there is a concordant and a discordant cycle, then we return to the Sublimation step in Procedure \ref{sublimation_step}. 
The Sublimation step removes a concordant cycle from $supp(f)$ and returns the rational decomposition of a circulation $h$ that is equivalent to $f$ and contains a discordant cycle. 
This $h$ restarts the loop in the role of $f$. Since $supp(h)$ contains fewer arcs than $supp(f)$, the Sublimation Loop is finite.

If all cycles are concordant in the rational decomposition returned by Proposition  \ref{real_cycle_decomp}, then applying the corollary of Proposition \ref{helly}
links the two loops of the procedure. 
This may be necessary if the attempted sublimation of a concordant cycle destroys all discordant ones,
which actually occurs in the circulation shown in Figure \ref{figExampleB}.
Corollary \ref{2cycles} reduces the circulation  to an equivalent circulation  
$g=\frac12 C_1 +\frac12 C_2$.  
cx 

\subsection{The $2$-cycle Loop}
If the Sublimation Loop does not find the desired non-negative integral linear combination, then it 
exits with a circulation $g$ equivalent to $f$ such that $W(g)=W(f)=W$, and $g=\frac12 C_1 +\frac12 C_2$, where $C_1$, $C_2$ are distinct concordant cycles and $supp(g)$ contains a discordant cycle. 
Resolving these types of circulations is the task of the $2$-cycle Loop.

The $2$-cycle Loop  consists of a procedure that proves the Proposition \ref{Case_d}.
The flow chart of the $2$-cycle Loop is presented in Figure \ref{2-cycle_Loop}. The algorithm reduces circulation $g$ 
through a sequence of equivalent circulations with a decreasing number of arcs. 
The reductions terminate producing a circulation that demonstrates the desired integer linear combination, possibly using arcs not found in $supp(g)$.

\begin{proposition}\label{Case_d} Let $g=\frac12 C_1 +\frac12 C_2$, where $C_1$, $C_2$ are distinct concordant cycles. If  $W(g)$ is a cycle inequality and $supp(g)$ contains a discordant cycle,
then there is an equivalent circulation $g^\prime$ that 
is the non-negative integral linear combination of discordant cycles. 
\end{proposition}
\begin{proof}
Lemma \ref{no_mixed} implies that $supp(g)$ has no mixed vertex.
Consequently, variables that appear in an inequality corresponding to a discordant cycle of $supp(g)$ must
correspond to vertices that are basic.  This means that a discordant cycle in $supp(g)$ corresponds to an inequality containing
a proper subset of the variables appearing in $W(g)$.
In particular, this means that $supp(g)$ (and therefore the concordant cycles $C_1$ and $C_2$) contain at least two basic vertices.
Notice that if $W$ is a $1$-variable redundant inequality, then its integral decomposition into discordant cycle inequalities is realized in the Sublimation Loop.  

Because $C_1\neq C_2$, there exists a vertex 
$\rho_w$ such that $C_1$ and $C_2$  `diverge' after both passing through $\rho_w$, that is, 
$\rho_w\black \rho_x$ belongs to $C_1$ and $\rho_w\black \rho_y$ belongs to $C_2$, and  $\rho_x\neq \rho_y$. The color of these divergent arcs is the same, 
because by  Lemma \ref{no_mixed}, $C_1\cup C_2$ has no mixed vertex.

	Case $1$:   $\rho_x\in C_2$. Let $\rho_v\black\rho_x$ be the arc in $C_2$. Because there are no mixed vertices, there are four possible colored configurations listed 
	in Figure \ref{x_in_C2}.   Observe that all of these vertices are necessarily basic.
			
For $\rho_v\red\rho_x$, the \FaceGraph contains the red arc  $\rho_v\red\rho_y$, by Lemma \ref{diamond}. The removal of the arcs $\rho_v\red\rho_x$ and $\rho_w\red\rho_y$ from $C_2$ leads to a partition of the vertices of  cycle $C_2$ into two disjoint cycles: $C_2^\prime$ beginning with the arc $\rho_w\red\rho_x$, and $C_2^{\prime\prime}$ beginning with the arc $\rho_v\red\rho_y$. Because $C_2$ is concordant and all of these vertices
are basic,  we obtain the required integer decomposition $g^\prime=C_2^\prime + C_2^{\prime\prime}$ into discordant cycles.
\begin{figure}[htp]
		\begin{center}
			\begin{tikzpicture}[scale=.8, ->,>=stealth',auto,node distance=3cm,
				thick,main node/.style={circle,draw,font=\sffamily\small\bfseries}, node/.style={}]
				
					\node[A,label=left:$\rho_w$] at (0,.75) (1) {};
				\node[A,label=above:$\rho_x$] at (1.25,1.5) (2) {};
				\node[A,label=below:$\rho_y$] at (1.25,0) (4) {};
				\node[A,label=left:$\rho_v$] at (0,2.25) (5) {};
				
				\path[draw, thick, every node/.style={font=\sffamily\small}]
				(1) edge [red] node [] {} (2)
				(1) edge [red,double,line width=.05em] node [] {} (4)
				(5) edge [red,double,line width=.05em] node [] {} (2)
				(5) edge  [red,dashed]  node [] {} (4);
			\end{tikzpicture}
			\hskip.5cm
			\begin{tikzpicture}[scale=.8, ->,>=stealth',auto,node distance=3cm,
				thick,main node/.style={circle,draw,font=\sffamily\small\bfseries}, node/.style={}]
				
					\node[A,label=left:$\rho_w$] at (0,.75) (1) {};
				\node[A,label=above:$\rho_x$] at (1.25,1.5) (2) {};
				\node[A,label=right:{$\rho_y=\rho_v$}] at (1.25,0) (4) {};
				
				\path[draw, thick, every node/.style={font=\sffamily\small}]
				(1) edge [red] node [] {} (2)
				(1) edge [red,double,line width=.05em] node [] {} (4)
				(4) edge [red,double,line width=.05em] node [] {} (2);
				(4) edge  [red]  node [] {} (4);
				 \draw[red,line width=.1em,dashed] (4) edge[out=200, in=-60, looseness=21] (4);
			\end{tikzpicture}
			\hskip.5cm
			\begin{tikzpicture}[scale=.8, ->,>=stealth',auto,node distance=3cm,
				thick,main node/.style={circle,draw,font=\sffamily\small\bfseries}, node/.style={}]
				
					\node[A,label=left:$\rho_w$] at (0,.75) (1) {};
				\node[A,label=above:$\rho_x$] at (1.25,1.5) (2) {};
				\node[A,label=below:$\rho_y$] at (1.25,0) (4) {};
				\node[A,label=left:$\rho_v$] at (0,2.25) (5) {};
				
				\path[draw, thick, every node/.style={font=\sffamily\small}]
				(1) edge [cyan] node [] {} (2)
				(1) edge [cyan,double,line width=.05em] node [] {} (4)
				(5) edge [cyan,double,line width=.05em] node [] {} (2)
				(5) edge  [cyan,dashed]  node [] {} (4);
			\end{tikzpicture}
			\hskip.5cm 
				\begin{tikzpicture}[scale=.8, ->,>=stealth',auto,node distance=3cm,
			thick,main node/.style={circle,draw,font=\sffamily\small\bfseries}, node/.style={}]
				
					\node[A,label=left:$\rho_w$] at (0,.75) (1) {};
				\node[A,label=above:$\rho_x$] at (1.25,1.5) (2) {};
				\node[A,label=left:{$$}] at (1.25,0) (4) {};
				\node[label=left:{$\rho_y=\rho_v(=\rho_z)$}] at (1.25,-0.25) () {};
				
				\path[draw, thick, every node/.style={font=\sffamily\small}]
				(1) edge [cyan] node [] {} (2)
				(1) edge [cyan,double,line width=.05em] node [] {} (4)
				(4) edge [cyan,double,line width=.05em] node [] {} (2);
			\draw[red,line width=.1em,dashed] (4) edge[out=270, in=20, looseness=15] (4);
			\end{tikzpicture}
			\end{center}
		\caption{ $\rho_w$ is a monochromatic divergence, and $\rho_x\in C_2$}
		\label{x_in_C2}
	\end{figure}
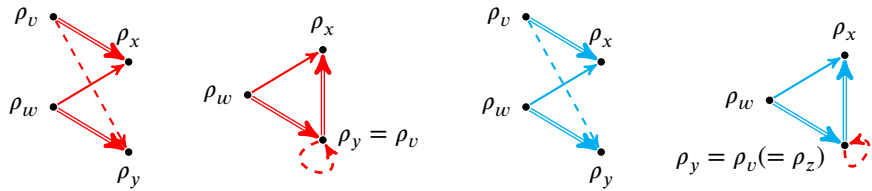
	
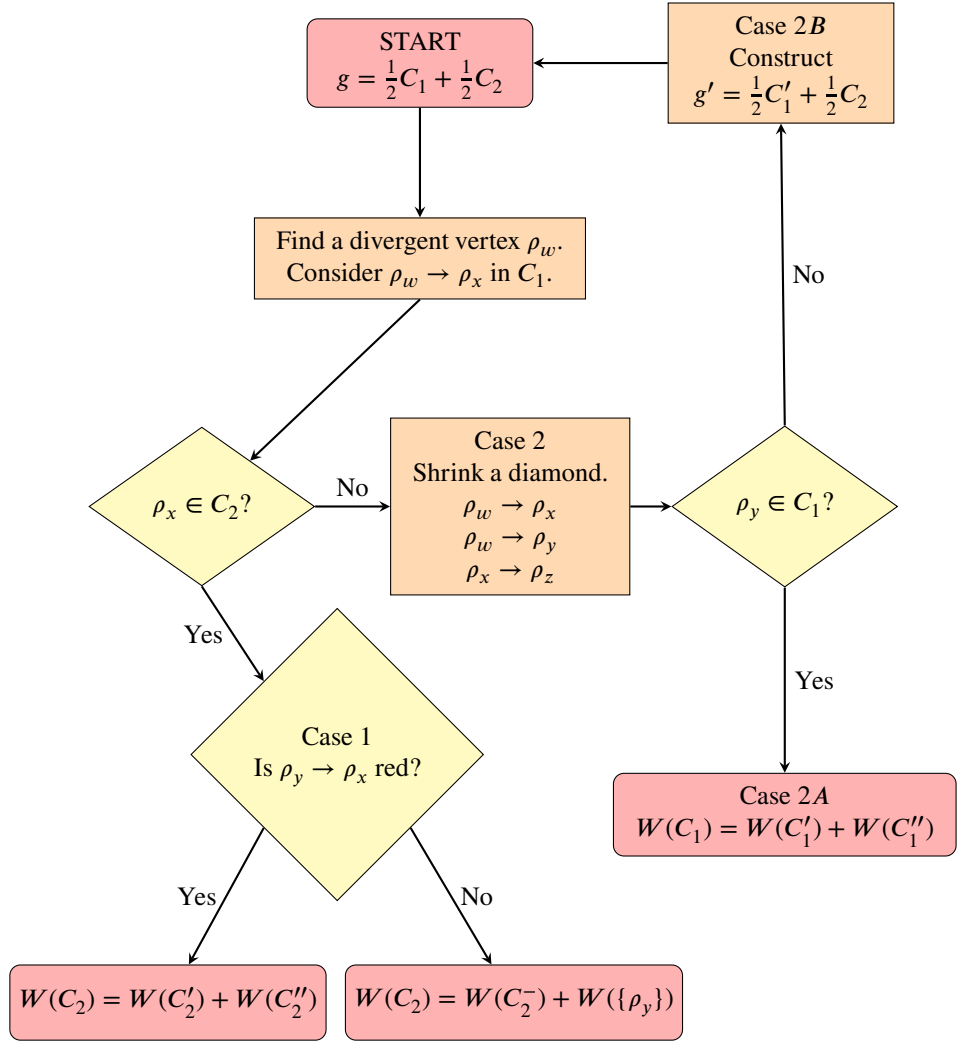
\begin{figure}[htp]
\begin{center}
\begin{tikzpicture}[node distance=1.6cm]
	\centering
	\node (start) [startstop] {$\begin{array}{c} \mbox{START} \\ g = \frac12 C_1 + \frac12 C_2\end{array}$};
	\node (process1) [process, below of=start, xshift=0cm, yshift=-1cm] {$\begin{array}{c} \mbox{Find a divergent vertex } \rho_w.\\ \mbox{Consider } \rho_w \rightarrow \rho_x \mbox{ in } C_1. \end{array}$};
	\node (decision1) [decision, below of=process1, xshift=-2.9cm, yshift=-1.7cm] {$\rho_x \in C_2$?};
	\node (decision2) [decision, below of=decision1, xshift=1.8cm, yshift=-1.7cm] {$\begin{array}{c} \mbox{Case }1 \\ \mbox{Is }\rho_y \rightarrow \rho_x  \mbox{ red?}\end{array}$ };
	\node (process2) [process, right of=decision1, xshift=2.5cm, yshift=0cm] {$\begin{array}{c} \mbox{Case }2\\ \mbox{Shrink a diamond.}\\ \rho_w \rightarrow \rho_x \\ \rho_w \rightarrow \rho_y \\ \rho_x \rightarrow \rho_z \end{array}$};
	\node (decomp1) [startstop, below of=decision2, xshift=-2.25cm, yshift=-1.7cm] {$W(C_2)=W(C_2^\prime)+W(C_2^{\prime\prime})$};
	\node (decomp2) [startstop, below of=decision2, xshift=+2.35cm, yshift=-1.7cm] {$W(C_2)=W(C_2^-)+W(\{\rho_y\})$};
	
	\node (decision3) [decision, right of=process2, xshift=2.05cm, yshift=-0cm] {$\rho_y \in C_1$?};
	\node (fin) [startstop, below of=decision3, yshift=-2.5cm] {$\begin{array}{c} \mbox{Case }2A \\ W(C_1)=W(C_1^\prime)+W(C_1^{\prime\prime})\end{array}$};
	
	\node (process3) [process, right of=start, xshift=3.2cm, yshift=0cm] {$\begin{array}{c} \mbox{Case }2B \\ \mbox{Construct} \\ g^\prime = \frac12 C_1^\prime + \frac12 C_2 \end{array}$};
	
	\draw [arrow] (start) -- (process1);
	\draw [arrow] (process1.south) -- (decision1);
	\draw [arrow] (decision1.south) -- node[anchor=east] {Yes} (decision2.north west);
	\draw [arrow] (decision1) -- node[anchor=south] {No} (process2);
	\draw [arrow] (decision2.south west) -- node[anchor=east] {Yes} (decomp1);
	\draw [arrow] (decision2.south east) -- node[anchor=west] {No} (decomp2);
	\draw [arrow] (decision3.south) -- node[anchor=west] {Yes} (fin.north);
	\draw [arrow] (process2) -- (decision3);
	\draw [arrow] (decision3.north) -- node[anchor=west] {No} (process3.south);
	\draw [arrow] (process3) -- (start);
	
\end{tikzpicture}
\end{center}
\caption{
A flowchart depicting the $2$-cycle Loop and the logical progression of the proof of Proposition \ref{Case_d}. 
The process begins with a cirulation $g$, defined as
the $\frac{1}{2}$-weighted sum of two concordant cycles exihibiting a redundant cycle inequality weight.  Termination occurs in one of three states, each
of which yields a $g$-equivalent non-negative integral linear combination of discordant cycles.}
\label{2-cycle_Loop}
\end{figure}
	
Notice that if $\rho_v=\rho_y$, then $C_2^{\prime\prime}$ becomes the red loop $\rho_v\red\rho_y$, thus in this case the decomposition of the equivalent circulation is  $g^\prime=(C_2-\rho_y) + (\rho_y\red\rho_y)$.

For $\rho_v\blue\rho_x$, Lemma \ref{diamond} implies $\rho_y\neq\rho_v$ when we apply it with
$\rho_w, \rho_x$ and $\rho_y=\rho_v=\rho_z$
 because there are no blue loops in $G_P$. Hence the blue triangle on the right of Figure \ref{x_in_C2} is actually not present in a \FaceGraph. A decomposition of $C_2$ is obtained in the same way as for the red divergence.

	Case $2$: $\rho_x\not\in C_2$. In this case either we obtain an integer decomposition of a concordant cycle (as in Case $1$) or we reduce the number of arcs of the circulation producing
	a new equivalent circulation with a discordant cycle that is the $\frac{1}{2}$-weighted sum of two concordant cycles.  This latter case we refer to as `shrinking' the circulation.
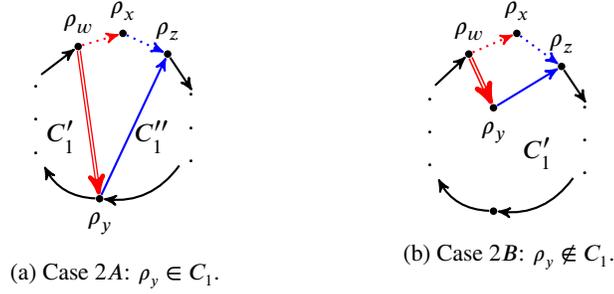
\begin{figure} [H]
		\begin{center}
			\begin{subfigure}[t]{0.05\textwidth}
				\centering
				$\ $
			\end{subfigure}
			\begin{subfigure}[T]{0.35\textwidth}
				\centering
		 		\begin{tikzpicture}[scale=1.1, ->,>=stealth',auto,node distance=3cm,
		 	main node/.style={circle,draw,font=\sffamily\small\bfseries}, node/.style={}]
		 			\centering
		 			\node[A,label=above:$\rho_w$] at (-0.54064082, 0.84125353) (1) {};
		 			\node[A,label=above:$\rho_x$] at (0,1) (2) {};
		 			\node[A,label=above:$\rho_z$] at (0.54064082, 0.75) (3) {};
		 			\node[A,label=below:$\rho_y$] at (-0.28173256, -1) (4) {};
		 			\node[node] at (0.75574957, -0.486073) (6) {};  
		 			\node[node] at (0.90963200, 0.2) (7) {};  
		 			\node[node] at (-1, -0.5) (8) {};  
		 			\node[node] at (-1.1, .41501) (9) {};  
		 	\node[]() at (0.35,-.25){$C_1^{\prime\prime}$};\node[]() at (-.75,-.25){$C_1^\prime$};
		 			\node[label=center:\rotatebox{90}{\bf $. \quad . \quad .$}] at (0.824, -0.14231484) {};
		 			\node[label=center:\rotatebox{90}{\bf $. \quad . \quad .$}] at (-1.05, -0.04231484) {};
		 			
		 			\path[draw, thick, every node/.style={font=\sffamily\small}]
		 			(3) edge [ black] node {} (7)
		 			(1) edge [red,dotted] node {} (2)
					(2) edge [blue,dotted] node  {} (3)
		 			(1) edge [red,double,line width=.05em] node {} (4)
		 			(4) edge [  blue] node {} (3)			
		 			(6) edge [ black,bend left] node {} (4)
		 			(9) edge [ black] node {} (1)
		 			(4) edge [ black,bend left] node {} (8);
		 			
		 		\end{tikzpicture}
		 	\caption{Case $2A$:  $\rho_y\in C_1$.}
			\end{subfigure}
			\begin{subfigure}[t]{0.05\textwidth}
				\centering
				$\ $
			\end{subfigure}
			\begin{subfigure}[T]{0.35\textwidth}
				\centering
					\begin{tikzpicture}[scale=1.1, ->,>=stealth',auto,node distance=3cm,
		 	main node/.style={circle,draw,font=\sffamily\small\bfseries}, node/.style={}]
		 			\centering
		 			\node[A,label=above:$\rho_w$] at (-0.58, 0.9) (1) {};
		 			\node[A,label=above:$\rho_x$] at (0,1.15) (2) {};
		 			\node[A,label=above:$\rho_z$] at (0.54064082, 0.75) (3) {};
		 			\node[A,,label=above:$ $] at (-0.28173256, -1) (4) {};
					\node[A,label=below:$\rho_y$] at (-.28173256, .25) (y) {};
		 			\node[node] at (0.75574957, -0.486073) (6) {};  
		 			\node[node] at (0.90963200, 0.2) (7) {};  
		 			\node[node] at (-1, -0.5) (8) {};  
		 			\node[node] at (-1.1, .41501) (9) {};  
		 			\node[]() at (0.25,-.3){$C_1^\prime$};
	
		 			\node[label=center:\rotatebox{90}{\bf $. \quad . \quad .$}] at (0.824, -0.14231484) {};
		 			\node[label=center:\rotatebox{90}{\bf $. \quad . \quad .$}] at (-1.05, -0.04231484) {};

		 			\path[draw, thick, every node/.style={font=\sffamily\small}]
		 			(3) edge [ black] node {} (7)
		 			(1) edge [red,dotted] node {} (2)
					(2) edge [blue,dotted] node  {} (3)
		 			(1) edge [double, red,line width=.06em] node {} (y)
		 			(y) edge [blue] node {} (3)
		 			(6) edge [ black,bend left] node {} (4)
		 			(9) edge [ black] node {} (1)
		 			(4) edge [ black,bend left] node {} (8);					
		 		\end{tikzpicture}	
		 	\caption{Case $2B$:  $\rho_y\not\in C_1$.}
		 	\end{subfigure}
		 	\end{center}
		 	\caption{The two cases of shrinking a diamond.}
			\label{shrinkredred}
		 \end{figure} 	
Let $\rho_x\black \rho_z$ belong to $C_1$.  Because $\rho_x \not\in C_2$, the arcs  $\rho_w\black \rho_x$ and  $\rho_x\black \rho_z$ have opposite color.
Consequently, by Lemma \ref{diamond}, the four (not necessarily distinct) vertices, $\rho_w,\rho_x,\rho_y,\rho_z$ form a diamond of the {\it \FaceGraph}. Namely, 
$(\rho_w,\rho_x,\rho_z)$ and $(\rho_w,\rho_y,\rho_z)$ are directed 3-paths  such that 
$\rho_w\black \rho_x$ and $\rho_w\black \rho_y$ have the same color, and  $\rho_x\black \rho_z$ and $\rho_y\black \rho_z$ have the same opposite color. 

{\t Shrinking the diamond} consists of 
removing  $(\rho_w,\rho_x,\rho_z)$  from $C_1$, 
and replacing it with   $(\rho_w,\rho_y,\rho_z)$. 
Notice that $\rho_w\black \rho_y$ is in $C_2$, and even if the arc $\rho_y\rightarrow \rho_z$ of the \FaceGraph  is not in $supp(g)$, shrinking reduces the number of arcs in the support. 
Also carefully note that, since $\rho_x$ is not basic, the new circulation has the same weight and still has a discordant cycle.

To complete the analysis, we discuss the case  $\rho_w\red \rho_x$ and $\rho_w\red \rho_y$;  
the case when both of these arcs are blue is essentially the same (details are not repeated). 
There are two cases to consider depending on whether $\rho_y$ is a vertex of $C_1$ or not (see Figure \ref{shrinkredred}).

Case $2A$:  $\rho_y\in C_1$. Let $C_1^\prime, C_1^{\prime\prime}$ be the two cycles sharing the vertex $\rho_y$ and replacing  $C_1$ after shrinking. We claim that the circulation $g^\prime=C_1^\prime +C_1^{\prime\prime}$ is equivalent to $g$, therefore,  both cycles  in the decomposition are discordant. The equivalence follows by checking that $\rho_y$ is a basic vertex in the decomposition if and only if it was a basic vertex in $C_1$. This is true in each of the four possible colorings passing through $\rho_y$ as shown in Figure \ref{fig:E} Thus we obtain  the required non-negative integral linear combination $W=W(g)=W(g^\prime)=W(C^\prime)+W(C^{\prime\prime})$.
		 
\begin{figure}[H]
		\begin{subfigure}[b]{0.22\textwidth}
			\centering
			\begin{tikzpicture}[scale=0.6, ->,>=stealth',auto,node distance=3cm,
				thick,main node/.style={circle,draw,font=\sffamily\small\bfseries}, node/.style={}]
				
				\node[node] at (2,2) (1) {};
				\node[node] at (4,2) (3) {};
				\node[A,label=below:$\rho_y$] at (3,-.2) (4) {};
				\node[node] at (5,1) (6) {};  
				\node[node] at (1,1) (7) {};  
				
				\path[draw,  thick, every node/.style={font=\sffamily\small}]
				
				(1) edge [ red] node {} (4)
				(4) edge [ blue] node {} (3)
				(6) edge [red,bend left] node {} (4)
				(4) edge [red,bend left] node {} (7);
			\end{tikzpicture}
			\caption{$\rho_y$ is basic}
			\label{fig:A}
		\end{subfigure}
		\begin{subfigure}[b]{0.22\textwidth}
			\centering
			\begin{tikzpicture}[scale=0.6, ->,>=stealth',auto,node distance=3cm,
				thick,main node/.style={circle,draw,font=\sffamily\small\bfseries}, node/.style={}]
				
				\node[node] at (2,2) (1) {};
				\node[node] at (4,2) (3) {};
					\node[A,label=below:$\rho_y$] at (3,-.2) (4) {};
				\node[node] at (5,1) (6) {};  
				\node[node] at (1,1) (7) {};  
				
				\path[draw, thick, every node/.style={font=\sffamily\small}]
				
				(1) edge [red] node {} (4)
				(4) edge [blue] node {} (3)
				(6) edge [red,bend left] node {} (4)
				(4) edge [blue,bend left] node {} (7);
			\end{tikzpicture}
			\caption{$\rho_y$ is not basic}
			\label{fig:B}
		\end{subfigure}
		\begin{subfigure}[b]{0.22\textwidth}
			\centering
			\begin{tikzpicture}[scale=0.6, ->,>=stealth',auto,node distance=3cm,
				thick,main node/.style={circle,draw,font=\sffamily\small\bfseries}, node/.style={}]
				
				\node[node] at (2,2) (1) {};
				\node[node] at (4,2) (3) {};
					\node[A,label=below:$\rho_y$] at (3,-.2) (4) {};
				\node[node] at (5,1) (6) {};  
				\node[node] at (1,1) (7) {};  
				
				\path[draw, thick, every node/.style={font=\sffamily\small}]
				
				(1) edge [red] node {} (4)
				(4) edge [blue] node {} (3)
				(6) edge [blue,bend left] node {} (4)
				(4) edge [red,bend left] node {} (7);
			\end{tikzpicture}
			\caption{$\rho_y$ is not basic}
			\label{fig:C}
		\end{subfigure}
		\begin{subfigure}[b]{0.22\textwidth}
			\centering
			\begin{tikzpicture}[scale=0.6, ->,>=stealth',auto,node distance=3cm,
				thick,main node/.style={circle,draw,font=\sffamily\small\bfseries}, node/.style={}]
				
				\node[node] at (2,2) (1) {};
				\node[node] at (4,2) (3) {};
					\node[A,label=below:$\rho_y$] at (3,-.2) (4) {};
				\node[node] at (5,1) (6) {};  
				\node[node] at (1,1) (7) {};  
				
				\path[draw, thick, every node/.style={font=\sffamily\small}]
				
				(1) edge [red] node {} (4)
				(4) edge [blue] node {} (3)
				(6) edge [blue,bend left] node {} (4)
				(4) edge [blue,bend left] node {} (7);
			\end{tikzpicture}
			\caption{$\rho_y$ is basic}
			\label{fig:D}
		\end{subfigure}
		\caption{The four possible colorings of arcs of $C_1^\prime$ and  $C_1^{\prime\prime}$.}
		\label{fig:E}
	\end{figure}
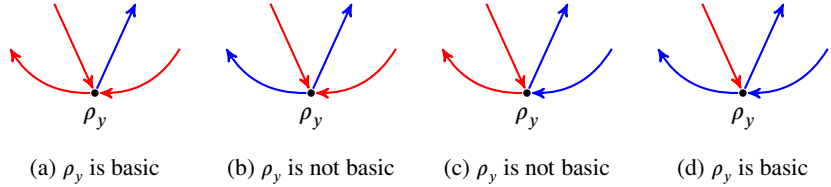	 

Case $2B$: $\rho_y \not\in C_1$.
In this case,
 $C_1^\prime$ is a concordant cycle,  hence $g^\prime = \frac12 C_1^\prime + \frac12 C_2$ is a circulation equivalent  to $g$. Since no weight of a cycle in $supp(g)$ changes,  $supp(g^\prime)$ has a discordant cycle. Observe that $g^\prime$ contains fewer arcs than $g$.  The algorithm loops back to the beginning of the $2$-cycle Loop. Thus, we obtain a circulation $g^\prime$ equivalent to $g$ and its  integral decomposition in a finite number of steps.
\end{proof}

\newcommand{\commentTikz}[1]{} 
\newcommand{\vertexScale}{0.6} 
\newcommand{\edgeLabelScale}{1.0} 
\definecolor{edgeBlue}{RGB}{0,0,255}
\definecolor{edgeRed}{RGB}{255,0,0}


\def\XAxisBump{.7} 
\def\MaxVerticalLine{8}  
\def\MinVerticalLine{0}  
\def\IntervalThickness{1.4}
\def\Gap{1.8} 
\def\VGap{.5}   
\def\Epsilon{.7}  
\def\Height{6} 

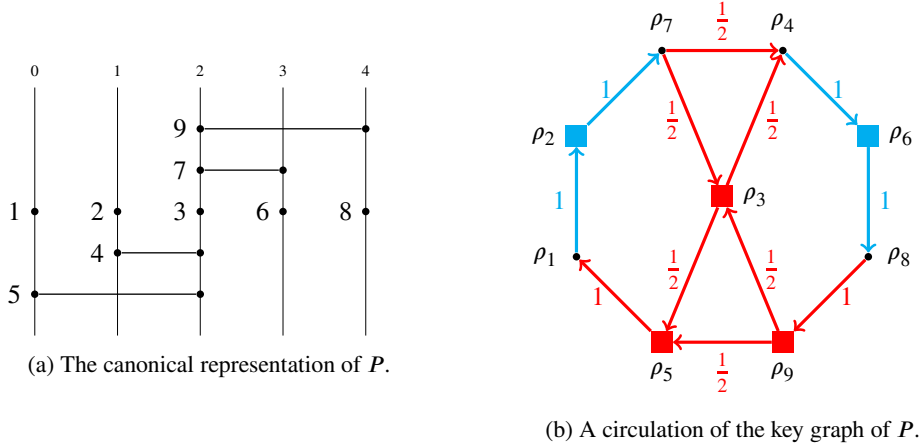
\begin{figure}[H]
\begin{center}
\begin{subfigure}[c]{0.45\textwidth}
\begin{tikzpicture}[scale=1.1]
	\centering
\foreach \i in {0,...,4}{
    \draw[line width=.02em](3+\i,3)--(3+\i,6);
 \node[]()at(3+\i,6.2){\tiny{\i}};   }
\node[A,label=left:9](L) at (5,5.5) {};
\node[A](R) at (7,5.5) {};
\draw[line width=.05em](L)--(R); 

\node[A,label=left:1](L) at (3,4.5) {};

\node[A,label=left:7](L) at (5,5) {};
\node[A](R) at (6,5) {};
\draw[line width=.05em](L)--(R); 

\node[A,label=left:5](L) at (3,3.5) {};
\node[A](R) at (5,3.5) {};
\draw[line width=.05em](L)--(R); 

\node[A,label=left:4](L) at (4,4) {};
\node[A](R) at (5,4) {};
\draw[line width=.05em](L)--(R);

\node[A,label=left:2]()at (4,4.5) {};
\node[A,label=left:3]()at (5,4.5) {};
\node[A,label=left:6]()at (6,4.5) {};
\node[A,label=left:8]()at (7,4.5) {};
\end{tikzpicture}
	\caption{The canonical representation of $P$. \label{figExampleA}}

\end{subfigure}
\hfill
\begin{subfigure}[c]{0.45\textwidth}
\begin{tikzpicture}[scale=1.4]
	\centering
    \def\radius{1.5}
    \foreach \i in {1,2,4,7} {
        \node[A] (v\i) at (22.5+\i*360/8:\radius){};
            }
\node[Yr,label=right:$\rho_3$] (v0) at (0,0){};
\node[label=above:$\rho_4$]  () at (v1){};
\node[label=above:$\rho_7$] ()  at (v2) {};
\node[Yb,label=left:$\rho_2$]  (v3) at (22.5+3*360/8:\radius) {};
\node[label=left:$\rho_1$]  () at (v4){};
\node[label=right:$\rho_8$]  () at (v7){};
\node[Yb,label=right:$\rho_6$]  (v8) at ((22.5+8*360/8:\radius){};
\node[Yr,label=below:$\rho_5$]  (v5) at (22.5+5*360/8:\radius){};
\node[Yr,label=below:$\rho_9$]  (v6) at (22.5+6*360/8:\radius){};


        \draw[red,line width=.12em,<-] (v1)--node [above] {$\frac12$} (v2); 
         \draw[cyan,line width=.12em,<-]  (v2)--node [left] {$1$} (v3); 
         \draw[cyan,line width=.12em,<-]  (v3)--node [left] {$1$} (v4);  
         \draw[red,line width=.12em,<-] (v4)--node [left] {$1$} (v5);
           \draw[red,line width=.12em,<-]  (v5)--node [below] {$\frac12$} (v6);  
           \draw[red,line width=.12em,<-] (v6)--node [right] {$1$} (v7);
             \draw[cyan,line width=.12em,<-]  (v7)--node [right] {$1$}(v8);  
             \draw[cyan,line width=.12em,<-]  (v8)--node [right] {$1$}(v1);
            \draw[red,line width=.12em,->] (v0)--node [right] {$\frac12$}  (v1); 
             \draw[red,line width=.12em,<-] (v0)--node [left] {$\frac12$}(v2);
             \draw[red,line width=.12em,->] (v0)--node [left] {$\frac12$}(v5);  
             \draw[red,line width=.12em,<-] (v0)--node [right] {$\frac12$} (v6);   
 \end{tikzpicture} 
	\caption{A circulation of the \FaceGraph of $P$. \label{figExampleB}}
\end{subfigure}  
\end{center}
\caption{A circulation whose weight is a redundant inequality that requires the $2$-cycle Loop to resolve. Removing
any concordant cycle destroys all discordant ones, so there is no heterogeneous decomposition.}
\label{figEx}
\end{figure}
 
Proposition \ref{Case_d}
closes the proof of Theorem \ref{integerNonNegativeLinearCombinationForInequalities}.  
Figure \ref{figEx} is an example showing a case where the $2$-cycle Loop obtains the required integral linear combination of a redundant cycle inequality. 
Consider the $9$ elements  interval order $P=(0, 1, 2, 1, 0, 3, 2, 4, 2)$.
The canonical representation of $P$ is shown on the left of Figure \ref{figEx}.

The redundant cycle inequality 
\begin{eqnarray} \label{ExampleInequality}
\left\langle 4 + \rho_{2} + \rho_{6}   \leq \rho_{3} + \rho_{5} + \rho_{9} \right\rangle
\end{eqnarray}
is the weight of the circulation $g$ defined on the right of Figure \ref{figEx}. Arcs of the two red triangles have flow $\frac12$, all other arcs have flow $1$. 
Note that $g$ is the $\frac{1}{2}$-weighted sum of concordant cycles: 
$C_1=(\rho_1,\rho_2,\rho_7,\rho_3,\rho_4,\rho_6,\rho_8,\rho_9,\rho_5)$  and $C_2=(\rho_1, \rho_2,\rho_7,\rho_4,\rho_6,\rho_8,\rho_9,\rho_3,\rho_5)$.

Ultimately, to show that cycle inequality (\ref{ExampleInequality}) is a non-negative integral linear
combination of smaller cycle inequalities, one may apply 
Lemma \ref{diamond} (III) with $\rho_w=\rho_3$ to deduce that arc $\rho_7\red \rho_{5}$
appears in the \FaceGraph of $P$.  Hence
cycle inequality (\ref{ExampleInequality}) is the sum of the cycle inequalities
arising from the cycle $C_2^{\prime}=(\rho_1.\rho_2,\rho_7,\rho_5)$ (with weight $2+\rho_{2} \leq \rho_{5}$) not in $supp(g)$ and from cycle $C_2^{\prime\prime}=(\rho_4,\rho_6,\rho_8,\rho_9,\rho_3)$
(with weight $2+\rho_{6} \leq\rho_3+ \rho_{9}$).

 This latter decomposition, $g=C_2^\prime+C_2^{\prime\prime}$, is given by Case $1$ of the $2$-cycle Loop  with $\rho_w=\rho_3,$ $ \rho_x= \rho_4$, $\rho_y=\rho_5,$ and $ \rho_v=\rho_7$.

\definecolor{edgeBlue}{RGB}{0,0,255}
\definecolor{edgeRed}{RGB}{255,0,0}


\def\XAxisBump{.7} 
\def\MaxVerticalLine{8}  
\def\MinVerticalLine{0}  
\def\IntervalThickness{1.4}
\def\Gap{1.8} 
\def\VGap{.5}   
\def\Epsilon{.7}  
\def\Height{6} 

\subsection{A weak order on cycle inequalities}
\label{weak}
In this subsection we introduce the following weak order on the cycle inequalities:
\begin{eqnarray} \label{WeakOrderOnInequalities}
	\left\langle \gamma_1  +\sum\limits_{i\in A_1}\rho_i  \leq  \sum\limits_{j\in B_1}\rho_j\right\rangle< \left\langle\gamma_2  +\sum\limits_{i\in A_2}\rho_i  \leq  \sum\limits_{j\in B_2}\rho_j\right\rangle
\end{eqnarray}
\begin{center}
if and only if
\end{center}
$$
(\gamma_1 < \gamma_2) \mbox{ or } (\gamma_1 = \gamma_2 \mbox{ and } |A_2| < |A_1|) \mbox{ or } (\gamma_1 = \gamma_2 \mbox{ and } |A_2| = |A_1| \mbox{ and } |B_1| < |B_2|).
$$

The weak order (\ref{WeakOrderOnInequalities}) is simply called {\em the weak order} on cycle inequalities.
We say an inequality is {\em smaller} or {\em larger} than another inequality in accordance with (\ref{WeakOrderOnInequalities}).  On the list of all cycle inequalities the smallest ones are the one-variable inequalities of the form $\left\langle \gamma \leq \rho_x \right\rangle$.
Recall that a cycle inequality is {\em redundant} if it is a non-negative linear combination of smaller cycle inequalities. The importance of this  ordering is highlighted in the Theorem \ref{WeakOrderSignificance} below.

Loop inequalities (inequalities of the form $\langle 0\leq \rho_x\rangle$) initialize the Schrijver list since they
are clearly irredundant cycle inequalities. We proved in Theorem \ref{integerNonNegativeLinearCombinationForInequalities} that redundant cycle inequalities are non-negative integral linear combinations of discordant cycle inequalities. In the next theorem, we show that these discordant components are smaller than the original redundant inequality in terms of the weak order.

\begin{theorem} \label{WeakOrderSignificance}
	A cycle inequality is redundant if and only if it is a non-negative integral linear combination of smaller inequalities.
\end{theorem}
\begin{proof}  If a cycle inequality is a non-negative linear combination of smaller inequalities, then it is redundant, by definition.
	Conversely, suppose that the cycle inequality
\begin{equation} \label{redundantCycleInequality}
		\left\langle\gamma  +\sum_{i\in A}\rho_i  \leq  \sum_{j\in B}\rho_j \right\rangle
\end{equation}
 is redundant.  This means it is not a loop inequality, so $\gamma>0$.  By Theorem \ref{integerNonNegativeLinearCombinationForInequalities},
the sum (\ref{redundantCycleInequality}) equals:
\begin{equation} \label{sumOfCycleInequalities}
\alpha_1 \left\langle \gamma_1  +\sum_{i\in A_1}\rho_i  \leq  \sum_{j\in B_1}\rho_j \right\rangle + \cdots + \alpha_t \left\langle \gamma_t  +\sum_{i\in A_t}\rho_i  \leq  \sum_{j\in B_t}\rho_j \right\rangle,
\end{equation}
where  $t\geq 2$ and $\alpha_1,\ldots,\alpha_t$ are positive integers. If $\gamma_i<\gamma$, then the corresponding $i$th inequality is smaller than $(\ref{redundantCycleInequality})$.  
If $\gamma_i=\gamma$, then for every $j\neq i$, $\alpha_j=0$ and $|A_j|=0$ thus  $|B_j|=1$. 
Then $|B_i|<|B|$, hence  the  $i$th inequality and all the other  inequalities in (\ref{sumOfCycleInequalities}) are smaller than (\ref{redundantCycleInequality}), as desired.
\end{proof}

Thanks to Theorem \ref{integerNonNegativeLinearCombinationForInequalities}, the weak order on cycle inequalities, and Theorem \ref{WeakOrderSignificance}, an algorithm is provided for identifying the irredundant inequalities from the full list of cycle inequalities.
 By Theorem \ref{UniqueMinimalTDI}, the remaining cycle inequalities form the unique minimal TDI subsystem, and thus the Schrijver system of the length polyhedron.

\section{Conclusions}

In practice, to compute the Schrijver system given by Theorem \ref{UniqueMinimalTDI} we have
used Johnson's algorithm \cite{Johnson} to enumerate all directed cycles in the \FaceGraph.  These cycles are converted to their corresponding inequalities (represented as vectors) and placed into a list realizing a linear extension of the weak order on cycle inequalities.
Detecting a redundant cycle inequality amounts to testing whether an inequality (vector) is a non-negative linear
combination of smaller irredundant inequalities (vectors); this is essentially a feasibility test for a linear program, which is 
reasonably efficient.  Unfortunately, the number of irredundant inequalities may become
exponential in the size of the interval order as the next example demonstrates.  


\def\N2{7}  
\def\N2Minus2{5}  
\def\n2{6}   
\def\g2{5}   
\def\h2{11} 
\def\e2{1.5}   

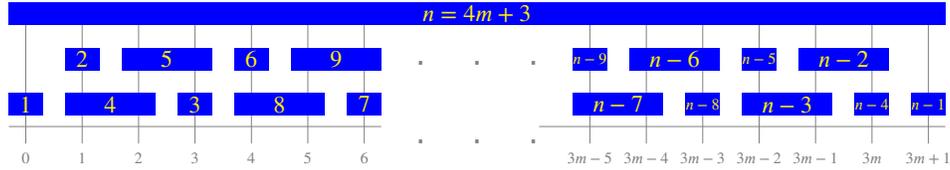
\begin{figure}[H]
	\begin{center}
		\begin{tikzpicture}[scale=0.15]
			\draw [gray, very thin] (-\e2,0)-- (\g2*\n2+\e2,0);
			
			\foreach \x in {0,...,\n2} {
				\draw [gray, very thin] (\g2*\x,\h2) -- (\g2*\x,-\e2) node[below] {{\tiny \x}};
			}
			
			\draw (\g2*\n2+\g2,-\e2) node{{\small \color{gray}$  \vysmblksquare$}};
			\draw (\g2*\n2+2*\g2,-\e2) node{{\small \color{gray}$  \vysmblksquare$}};
			\draw (\g2*\n2+3*\g2,-\e2) node{{\small \color{gray}$  \vysmblksquare$}};

			\draw (\g2*\n2+\g2,\h2/2) node{{\small \color{gray}$  \vysmblksquare$}};
			\draw (\g2*\n2+2*\g2,\h2/2) node{{\small \color{gray}$  \vysmblksquare$}};
			\draw (\g2*\n2+3*\g2,\h2/2) node{{\small \color{gray}$  \vysmblksquare$}};

			\draw [gray, very thin] (\g2*\n2+4*\g2-3*\e2,0)-- (\g2*\n2+4*\g2+\g2*\n2+\e2,0);
			
			\foreach \x in {3,...,7} {
\draw [gray, very thin] (\g2*\n2+5*\g2+\g2*\n2-\g2*\x,\h2) -- (\g2*\n2+5*\g2+\g2*\n2-\g2*\x,-\e2) node[below] {{\tiny $3m- \number\numexpr\x-2\relax$}};
			}
			
			\foreach \x in {2,...,2} {
	\draw [gray, very thin] (\g2*\n2+5*\g2+\g2*\n2-\g2*\x,\h2) -- (\g2*\n2+5*\g2+\g2*\n2-\g2*\x,-\e2) node[below] {{\tiny $3m$}};
}
			
						\foreach \x in {1,...,1} {
				\draw [gray, very thin] (\g2*\n2+5*\g2+\g2*\n2-\g2*\x,\h2) -- (\g2*\n2+5*\g2+\g2*\n2-\g2*\x,-\e2) node[below] {{\tiny $3m+1$}};
			}
			
			\fill[blue] (-\e2,\h2-2) rectangle (\g2*\n2+5*\g2+\g2*\n2-\g2+\e2,\h2);
			\draw (\g2*\n2+2*\g2,\h2-1) node{{\small \color{yellow}$n = 4m+3$}};

			\fill[blue] (0-\e2,1) rectangle (0+\e2,3);
			\draw (0,2) node{{\small \color{yellow}$1$}};	
			
			\fill[blue] (\g2-\e2,5) rectangle (\g2+\e2,7);
			\draw (\g2,6) node{{\small \color{yellow}$2$}};	
			
			\fill[blue] (3*\g2-\e2,1) rectangle (3*\g2+\e2,3);
			\draw (3*\g2,2) node{{\small \color{yellow}$3$}};	
			
			\fill[blue] (\g2-\e2,1) rectangle (2*\g2+\e2,3);
			\draw (1.5*\g2,2) node{{\small \color{yellow}$4$}};	
			
			\fill[blue] (2*\g2-\e2,5) rectangle (3*\g2+\e2,7);
			\draw (2.5*\g2,6) node{{\small \color{yellow}$5$}};	
			
			\fill[blue] (\g2+3*\g2-\e2,5) rectangle (\g2+3*\g2+\e2,7);
			\draw (\g2+3*\g2,6) node{{\small \color{yellow}$6$}};	
			
			\fill[blue] (3*\g2+3*\g2-\e2,1) rectangle (3*\g2+3*\g2+\e2,3);
			\draw (3*\g2+3*\g2,2) node{{\small \color{yellow}$7$}};	
			
			\fill[blue] (\g2+3*\g2-\e2,1) rectangle (2*\g2+3*\g2+\e2,3);
			\draw (1.5*\g2+3*\g2,2) node{{\small \color{yellow}$8$}};	
			
			\fill[blue] (2*\g2+3*\g2-\e2,5) rectangle (3*\g2+3*\g2+\e2,7);
			\draw (2.5*\g2+3*\g2,6) node{{\small \color{yellow}$9$}};

			\fill[blue] (\g2*\n2+5*\g2+\g2*\n2-\g2-\e2,1) rectangle (\g2*\n2+5*\g2+\g2*\n2-\g2+\e2,3);
			\draw (\g2*\n2+5*\g2+\g2*\n2-\g2,2) node{{\tiny \color{yellow}$n-1$}};	
			
			
			\fill[blue] (\g2*\n2+3*\g2+\g2*\n2-\g2-\e2,5) rectangle (\g2*\n2+4*\g2+\g2*\n2-\g2+\e2,7);
			\draw (\g2*\n2+3.5*\g2+\g2*\n2-\g2,6) node{{\small \color{yellow}$n-2$}};	
			
			\fill[blue] (\g2*\n2+2*\g2+\g2*\n2-\g2-\e2,1) rectangle (\g2*\n2+3*\g2+\g2*\n2-\g2+\e2,3);
			\draw (\g2*\n2+2.5*\g2+\g2*\n2-\g2,2) node{{\small \color{yellow}$n-3$}};	
			
			\fill[blue] (\g2*\n2+4*\g2+\g2*\n2-\g2-\e2,1) rectangle (\g2*\n2+4*\g2+\g2*\n2-\g2+\e2,3);
			\draw (\g2*\n2+4*\g2+\g2*\n2-\g2,2) node{{\tiny \color{yellow}$n-4$}};	
			
			\fill[blue] (\g2*\n2+2*\g2+\g2*\n2-\g2-\e2,5) rectangle (\g2*\n2+2*\g2+\g2*\n2-\g2+\e2,7);
			\draw (\g2*\n2+2*\g2+\g2*\n2-\g2,6) node{{\tiny \color{yellow}$n-5$}};	
			
			
			\fill[blue] (\g2*\n2+\g2*\n2-\g2-\e2,5) rectangle (\g2*\n2+\g2+\g2*\n2-\g2+\e2,7);
			\draw (\g2*\n2+0.5*\g2+\g2*\n2-\g2,6) node{{\small \color{yellow}$n-6$}};	
			
			\fill[blue] (\g2*\n2-1*\g2+\g2*\n2-\g2-\e2,1) rectangle (\g2*\n2+\g2*\n2-\g2+\e2,3);
			\draw (\g2*\n2-0.5*\g2+\g2*\n2-\g2,2) node{{\small \color{yellow}$n-7$}};	
			
			\fill[blue] (\g2*\n2+1*\g2+\g2*\n2-\g2-\e2,1) rectangle (\g2*\n2+1*\g2+\g2*\n2-\g2+\e2,3);
			\draw (\g2*\n2+1*\g2+\g2*\n2-\g2,2) node{{\tiny \color{yellow}$n-8$}};	
			
			\fill[blue] (\g2*\n2-1*\g2+\g2*\n2-\g2-\e2,5) rectangle (\g2*\n2-1*\g2+\g2*\n2-\g2+\e2,7);
			\draw (\g2*\n2-1*\g2+\g2*\n2-\g2,6) node{{\tiny \color{yellow}$n-9$}};					
			
		\end{tikzpicture}
		\caption{\label{ExponentialSizeTDI} The canonical representation of (twin-free) interval orders whose length polyhedra  
			have unique minimal TDI-systems that grow exponentially.}
	\end{center}
\end{figure}


Consider the interval orders whose
canonical representations are depicted in Figure \ref{ExponentialSizeTDI}.
Let $P_m$ denote the interval order in this figure, where $m$ is a positive integer.
The canonical representation of $P_m$ contains $2m+2$ intervals of length zero and
$2m$ intervals of length one.  These intervals correspond to $4m+2$ irredundant cycle inequalities of the length polyhedron.
There are also $3^m$ irredundant cycle inequalities derived from cycles of the \FaceGraph that take
the form $(1,a_1,a_2,\ldots,a_{2m},a_{2m+1},n-1,n)$, where $(a_{2i},a_{2i+1})$ is one of three
pairs from $\{(4i-2,4i-1),(4i-2,4i+1),(4i,4i-1)\}$, for $1 \leq i \leq m$.
After some work checking the $5^m + 4m + 2$ directed cycles in the \FaceGraph (which we omit),
one can verify that these are all of the irredundant cycle inequalities; so,
the number of inequalities in the Schrijver system given by Theorem \ref{UniqueMinimalTDI} for $P_m$
has exactly $3^m + 4m + 2$ inequalities.

It would be interesting to characterize interval orders whose \FaceGraph have bounded circumference (say circumference
less than or equal to $k$ for some positive integer $k$)
as this would guarantee a polynomial-sized Schrijver system determining the length polyhedron.
In a related computational-complexity question, one wonders whether testing the irreducibility of 
a cycle inequality can be done in time polynomial in the size of the interval order, regardless of the total
number of irreducible cycle inequalities ultimately.  

\subsection*{Acknowledgments}

We thank Csaba Bir\'o for stimulating discussions.

\bibliographystyle{abbrv}
\bibliography{TheReferences}

\end{document}